\documentclass[10pt]{article}
\usepackage[letterpaper,margin=1in]{geometry}

\usepackage{amsmath,amsthm}
\usepackage{amsbsy}
\usepackage{amsfonts}
\usepackage{amssymb}
\usepackage{amscd}
\usepackage{mathrsfs}
\usepackage{bm}
\usepackage{algpseudocode}

\usepackage{txfonts}
\usepackage{lmodern}

\usepackage[latin1]{inputenc}
\usepackage{graphicx}
\usepackage{subfig} 
\usepackage{float}

\usepackage{hyperref}
\usepackage{type1cm}       

\usepackage{soul}
\usepackage{url}

\usepackage{rotating} 
\usepackage{makeidx}
\makeindex             
\usepackage{multicol}   
\usepackage{enumerate}
\usepackage{xspace}
\usepackage{comment}  
\usepackage{etoolbox}
\usepackage{fancyhdr}

\usepackage{array}

\newcommand{\dnn}{{}_n{\mathcal N}}

\numberwithin{equation}{section}
\numberwithin{figure}{section}
\numberwithin{table}{section}

\ifcsmacro{theorem}{}{
\newtheorem{theorem}{Theorem}[section]
}
\ifcsmacro{lemma}{}{
\newtheorem{lemma}[theorem]{Lemma}
}
\ifcsmacro{corollary}{}{

}
\ifcsmacro{proposition}{}{

}
\ifcsmacro{algorithm}{}{

}

\newtheorem{remark}[theorem]{Remark}

\newcommand{\balpha}{{\boldsymbol{\alpha}}}

\newcommand{\relu}{\mbox{ReLU}}

\begin{document}

\title{The Finite Neuron Method and Convergence Analysis}
\author{
Jinchao Xu\footnote{xu@math.psu.edu, Department of Mathematics,
Pennsylvania State University, University Park, PA, 16802, USA}
}
\date{August 2020}

\maketitle

\begin{abstract}
  We study a family of $H^m$-conforming piecewise polynomials based on
  artificial neural network, named as the {\it finite neuron method}
  (FNM), 
  for numerical solution of $2m$-th order partial differential
  equations in $\mathbb{R}^d$ for any $m,d \geq 1$ and then provide
  convergence analysis for this method.  Given a general domain
  $\Omega\subset\mathbb R^d$ and a partition $\mathcal
  T_h$ of $\Omega$, it is still an open problem in general how to
  construct conforming finite element subspace of $H^m(\Omega)$ that
  have adequate approximation properties. By using techniques from
  artificial neural networks, we construct a family of
  $H^m$-conforming set of functions consisting of piecewise
  polynomials of degree $k$ for any $k\ge m$ and we further obtain the
  error estimate when they are applied to solve elliptic boundary
  value problem of any order in any dimension.  For example, the
  following error estimates between the exact solution $u$ and finite
  neuron approximation $u_N$ are obtained. 
$$
\|u-u_N\|_{H^m(\Omega)}=\mathcal O(N^{-{1\over 2}-{1\over d}}).
$$
Discussions will also be given on the difference and relationship
between the finite neuron method and finite element methods (FEM).  For
example, for finite neuron method, the underlying finite element grids
are not given a priori and the discrete solution can only be obtained
by solving a non-linear and non-convex optimization problem.  
Despite of  many desirable theoretical properties of the finite
neuron method analyzed in the paper, its practical value is a subject
of further investigation since the aforementioned underlying
non-linear and non-convex optimization problem can be expensive and
challenging to solve.  For completeness and also convenience to
readers, some basic known results and their proofs are also included
in this manuscript.
\end{abstract}

\tableofcontents
\section{Introduction}
This paper is devoted to the study of numerical methods for high order
partial differential equations in any dimension using appropriate
piecewise polynomial function classes.   In this introduction, we will
briefly describle a class of elliptic boundary value problems of any order
in any dimension, we will then give an overview of some existing
numerical methods for this model and other related problems, and we
will finally explain the motivation and objective of this paper. 
\subsection{Model problem}
Let $\Omega\subset \mathbb{R}^d$ be a bounded domain with a
sufficiently smooth boundary $\partial\Omega$.  For any integer $m\ge
1$, we consider the following model $2m$-th order partial differential
equation with certain boundary conditions:
\begin{equation} \label{2mPDE}
\left\{
  \begin{array}{rccl}\displaystyle
Lu &=& f &\mbox{in }\Omega, \\
B^k(u) &= &0 & \mbox{on }\partial\Omega \quad(0\le k\le m-1),
  \end{array}
\right.
\end{equation}
where $L$ is a partial differential operator as follows
\begin{equation}\label{Lu}
Lu= \sum_{|\alpha|=m}(-1)^m\partial^\alpha (a_\alpha(x)\,\partial^\alpha\,u) +a_0(x)u,
 \end{equation} 
and $\balpha$ denotes $n$-dimensional multi-index $\balpha
= (\alpha_1, \cdots, \alpha_n)$ with
$$
|\balpha| = \sum_{i=1}^n \alpha_i, \quad
\partial^\balpha = \frac{\partial^{|\balpha|}}{\partial x_1^{\alpha_1}
\cdots \partial x_n^{\alpha_n}}.
$$
For simplicity, we assume that $a_\alpha$ are strictly positive and smooth functions on
$\Omega$ for  $|\alpha|=m$ and $\alpha=0$, namely, $\exists \alpha_0>0$, such that 
\begin{equation}\label{ass:1}
a_\alpha(x), a_0(x)\ge \alpha_0,\,\, \forall x\in\Omega,\,\,|\alpha|=m.
 \end{equation} 
Given a nonnegative integer
$k$ and a bounded domain $\Omega\subset \mathbb{R}^d$, let 
$$
H^k(\Omega):=\left\{v\in L^2(\Omega), \partial^\alpha v\in L^2(\Omega), |\alpha|\le k\right\}
$$
be standard Sobolev spaces with norm and seminorm given respectively by 
$$
 \|v\|_k:=\left(\sum_{|\alpha|\le k} \|\partial^\alpha v\|_0^2\right)^{1/2}, \quad  |v|_k:=\left(\sum_{|\alpha|= k} \|\partial^\alpha v\|_0^2\right)^{1/2}.
$$
For $k=0$, $H^0(\Omega)$ is the standard $L^2(\Omega)$ space
with the inner product denoted by $(\cdot, \cdot)$.  Similarly, for
any subset $K\subset \Omega$, $L^2(K)$ inner product is denoted by
$(\cdot, \cdot)_{0,K}$.  We note that, under the assumption \eqref{ass:1},
\begin{equation}
  \label{avv}
a(v,v)\gtrsim \|v\|^2_{m,\Omega}, \forall v\in H^m(\Omega).
\end{equation}

The boundary value problem \eqref{2mPDE} can be cast into an
equivalent optimization or a variational problem as described below for some approximate subspace $V\subset H^m(\Omega)$.
\begin{description}
\item[Minimization Problem M:] Find $u\in V$ such that 
\begin{equation}
\label{minJv}
J(u)=\min_{v\in V} J(v),
\end{equation}
\noindent or
\item[Variational Problem V:]  Find  $u\in V$ such that 
\begin{equation}\label{m-vari} 
a(u,v) = \langle f, v\rangle \quad \forall v \in V.
\end{equation}
\end{description}
The bilinear form $a(\cdot,\cdot)$ in \eqref{m-vari}, the objective
function $J(\cdot)$ in \eqref{minJv} and the functional space $V$
depend on the type of boundary condition in \eqref{2mPDE}. 

One popular type of boundary conditions are Dirichlet boundary
condition when $B^k=B_D^k$ are given by the following Dirichlet type
trace operators
\begin{equation}\label{BD}
B_D^k(u):=\left.\frac{\partial^k u}{\partial
    \nu^k}\right|_{\partial\Omega}\quad (0\le k\le m-1),
\end{equation}
with $\nu$ being the outward unit normal vector of $\partial\Omega$. 

For the aforementioned Dirichlet boundary condition, the elliptic
boundary value problem \eqref{2mPDE} is equivalent to \eqref{minJv} or
\eqref{m-vari} with $V=H^m_0(\Omega)$ and 
\begin{equation}
  \label{auv}
a(u,v) := \sum_{|\alpha | = m}(a_\alpha\partial^{\alpha}u, \partial^{\alpha}v)_{0,\Omega} +(a_0u,v)\quad \forall
u, v \in V,
\end{equation}
and 
\begin{equation}
  \label{Jv}
J(v)=\frac12 a(v,v) -\int_{\Omega} fv dx.  
\end{equation}

Other boundary conditions such as Neumann boundary and mixed boundary
conditions are a little bit complicated to describe for general case
when $m\ge 2$ and will be discussed later.

\subsection{A brief overview of existing methods}
Here we briefly review some classic finite element and other
relevant methods for numerical solution of elliptic boundary value problems \eqref{2mPDE} for all $d, m\ge
1$.

Classic finite element methods use piecewise polynomial functions
based on a given a subdivision, namely a finite element grid, of the
domain, to discretize the variational problem.  We will mainly review
three different types of finite element methods: (1)
conforming element method; (2) nonconforming and discontinuous
Galerkin method; and (3) virtual element method. 

\paragraph{Conforming finite element method.}  Given a finite element
grid, this type of method is to construct $V_h\subset V$ and find
$u_h\in V_h$ such that
\begin{equation}
J(u_h)=\min_{v_h\in V_h} J(v_h).
\end{equation}
It is well-known that  a piecewise polynomial $V_h\subset H^m(\Omega)$ if and only if 
$V_h\subset C^{m-1}(\bar\Omega)$.  For $m=1$, piecewise linear finite element $V_h\subset H^1(\Omega)$ can be easily constructed on simplicial finite element grids in any dimension $d\ge 1$.  The construction and analysis of linear finite element method for $m=1$ and $d=2$ can be traced back to \cite{feng1965finite}. The situation becomes  complicated when $m \geq 2$ and $d \geq 2$. 

For example, it was proved that the construction of an $H^2$-conforming finite element space requires the use of 
polynomials of at least degree five in two dimensions \cite{vzenivsek1970interpolation} and degree nine in three dimensions \cite {lai2007spline}.
We refer to \cite{argyris1968tuba} for the classic quintic $H^2$-Argyris element in two dimension and to \cite{zhang2009family} for the ninth-degree $H^2$-element in three dimensions.  

Many other efforts have been made in the literature in constructing $H^m$-conforming finite element spaces. \cite{bramble1970triangular} proposed the 2D simplicial $H^m$ conforming elements ($m\geq 1$) by using the polynomial spaces of degree $4m-3$, which are the generalization of the $H^2$ Argyris element (cf.\,\cite{argyris1968tuba, ciarlet1978finite}) and $H^3$ {\v{Z}}en{\'\i}{\v{s}}ek element (cf.\,\cite{vzenivsek1970interpolation}). Again, the degree of polynomials used is quite high.  For \eqref{2mPDE} an alternative in 2D is to use mixed methods based on the Helmholtz decompositions for tensor-valued functions 
(cf.\,\cite{schedensack2016new}). However, the general construction of $H^m$-conforming elements in any dimension is still an open problem. 

We note that the construction of conforming finite element space
depends on the structure of the underlying grid. For example, one can
construct relatively low-order $H^2$ finite elements on grids with
special structures. Examples include the (quadratic) Powell-Sabin,
(cubic) Clough-Tocher elements in two dimensions
\cite{powell1977piecewise,clough1965finite}, and the (quintic) Alfeld
splits in three dimensions \cite{alfeld1984trivariate}, where
full-order accuracy, namely $\mathcal{O}(h)$, $\mathcal{O}(h^2)$ and
$\mathcal{O}(h^4)$ accuracy can be estimated.  On more recent
developments on Alfeld splits we refer to \cite{fu2020exact} and
references cited therein.  But these constructions do not apply to
general grids. For example, de Boor-DeVore-H\"ollig
\cite{1983Approximationbv,1983Approximationbh} showed that the $H^2$
element that consists of piecewise cubic polynomials on uniform grid
sequence would not provide full approximation accuracy. This gives us
hints the structure of the underlying grid plays an important role in
constructing $H^m$-comforming finite element.

\paragraph{Nonconforming finite element and discontinuous Galerkin methods:}
Given a finite element grid $\mathcal T_h$, compared to conforming method, the nonconforming finite element method does not require that $V_h\subset V$, namely $V_h\nsubseteq V$. We find $u_h\in V_h$ such that
\begin{equation}
J_h(u_h)=\min_{v_h\in V_h} J_h(v_h)
\end{equation}
with 
$$\displaystyle J_h(v_h)=\sum_{K\in \mathcal T_h}J_K(v_h)=\sum_{K\in \mathcal T_h}\frac{1}{2}\int_{K}\sum_{|\alpha|=m} a_\alpha |\partial^\alpha v_h|^2+a_0|v_h|^2dx-\int_{K} fv_hdx.
$$

One interesting example of nonconforming element for \eqref{2mPDE} is the Morley element \cite{morley1967triangular} for $m=d=2$  which uses 
piecewise quadratic polynomials. For $m\le d$, Wang and Xu \cite{wang2013minimal} provided a universal construction and analysis for a family of nonconforming finite elements consisting piecewise polynomials of minimal order for \eqref{2mPDE} on
$\mathbb{R}^d$ simplicical grids. The elements in \cite{wang2013minimal}, now known as MWX-elements in the literature, 
gave a natural generalization of the classic Morley element to the general case that $1\le m\le d$. 
Recently, there are a number of results on the extension of MWX-elements.
\cite{wu2017nonconforming} enriched the $\mathcal{P}_m$
polynomial space by $\mathcal{P}_{m+1}$ bubble functions to obtain a
family of $H^{m}$ nonconforming elements when $m=d+1$.  \cite{hu2016canonical} applied the full $\mathcal{P}_4$
polynomial space for the construction of nonconforming element when
$m=3, d=2$, which has three more degrees of freedom locally than the
element in \cite{wu2017nonconforming}.  They also used the full
$\mathcal{P}_{2m-3}$ polynomial space for the nonconforming finite
element approximations when $m \geq 4, d=2$. 

In addition to the aforementioned conforming and nonconforming finite element methods, discontinuous Galerkin (DG) method that make use of piecewise polynomials but globally discontinuous finite element functions have been also used for solving high order partial differential equations, c.f. \cite{baker1977finite}. The DG method requires the use of many stabilization terms and parameters and the number of stabilization terms and parameters naturally grow as the order of PDE grows.  To reduce the amount of stabilization, one approach is to introduce some continuity and smoothness in the discrete space to replace the totally discontinuous spaces.  Examples for such an approach include the $C^0$-interior penalty DG methods for fourth order elliptic problem by Brenner and Sung \cite{brenner2005c0} and for sixth order elliptic equation by Gudi and Neilan \cite{ gudi2011interior}. More recently, Wu and Xu \cite{wu2017pm} provided a family of interior penalty nonconforming finite element methods for \eqref{2mPDE} in $\mathbb R^d$, for any $m\ge 0, d\ge 1$. This family of elements recover the MWX-elements in \cite{wang2013minimal} when $m\le d$ which does not require any stabilization.

\paragraph{Virtual finite element}   Classic definition of finite element methods \cite{ciarlet1978finite} based on finite element triple can be extended in many different ways.  One successful extension is the virtual element method (VEM) in which general polygons or polyhedrons are used as elements and non-polynomial functions are used as shape functions.  For $m=1$, we refer to \cite{beirao2013basic} and \cite{brezzi2014basic}.  For $m=2$, we refer to \cite{brezzi2013virtual} on conforming virtual element methods for plate bending problems, and  \cite{antonietti2018fully} on nonconforming virtual element methods for biharmonic problems.  For general $m\ge 1$, we refer to  \cite{chen2020nonconforming} for nonconforming elements which extend the MWX elements in \cite{wang2013minimal} from simplicial elements to polyhedral elements. 

\subsection{Objectives}
Deep neural network (DNN), a tool developed for machine learning
\cite{goodfellow2016deep}.  DNN provides a very special function class
that have been used for numerical solution of partial different
equations, c.f. ~\cite{lagaris1998artificial}. By using different activation
function such as sigmoidal, deep neural network can give rise to a
very wide range of functional classes that can be drastically different
from the piecewise polynomial function classes used in classic finite
element method.  One advantage of DNN approach is that it is quite easy to
obtain smooth, namely $H^m$-conforming for any $m\ge1$, DNN functions
by simply choosing smooth activation functions.  These function
classes, however, do not usually form a linear vector space and hence
the usual variational principle in classic finite element method can
not be applied easily and instead collocation type of methods are
often used.   DNN is known to have much less ``curse of
dimensionality'' than the traditional functional classes (such as
polynomials or piecewise polynomials), DNN based method is potentially
efficient to high dimensional problems and has been studied, for
example, in \cite{weinan2017deep} and \cite{sirignano2018dgm}.

One main motivation of this paper is to explore DNN type methods that
are most closely related to the traditional finite element methods.
Namely we are interested in DNN function classes that consist of
piecewise polynomials.  By exploring relationship between DNN and FEM,
we hope, on one hand, to expand or extend the traditional FEM approach
by using new tools from DNN, and, on the other hand, to gain and
develop theoretical insights and algorithmic tools into DNN by
combining the rich mathematical theories and techniques in FEM.

In an earlier work \cite{he2020relu}, we studied the relationship
between deep neural networks (DNNs) using ReLU as activation
function and continuous piecewise linear functions.  One conclusion
that can be drawn from \cite{he2020relu} is that any ReLU-DNN function
is an $H^1$-conforming linear finite element function, and verse
versa.  The current work can be considered an extension of
\cite{he2020relu} by considering using ReLU$^k$-DNN for high order
partial differential equations.  One focus in the current work is to
provide error estimates when ReLU$^k$-DNN is applied to solve high
order partial differential equations.  More specifically, we will
study a special class of $H^m$-conforming generalized finite element
methods (consisting of piecewise polynomials) for \eqref{2mPDE} for
any $m\ge 1$ and $d\ge 1$ based on artificial neural network for numerical solution of
arbitrarily high order elliptic boundary value problem \eqref{2mPDE}
and then provide convergence analysis for this method. For this type of method, the
underlying finite element grids are not given a priori and the
discrete solution can be obtained for solving a non-linear and
non-convex optimization problem.  In the case that the boundary of
$\Omega\subset\mathcal R^d$, namely $\partial \Omega$, is curved, it
is often an issue how to put a good finite element grid to accurately
approximate $\partial\Omega$.  As it turns out, this is not an issue
for the finite neuron method, which is probably one of the advantages
of the finite neuron method analyzed in this paper.

We note that the numerical method studied in this paper for elliptic boundary value problems is closely related to the classic finite element method, namely it amounts to piecewise polynomials with respect to an implicitly defined grid. We can also argue that it can be viewed as a mesh-less or even vertex-less method. But comparing with the popular meshless method, this method does correspond to some underlying grid, but this grid is not given a priori. This underlying grid is determined by the artificial neurons, which mathematically speaking refers to hyperplanes $\omega_i\cdot x+b_i=0$, together with a given activation function. By combining the names for finite element method and artificial neural network, for convenience of exposition, we will name the method studied in this paper as the finite neuron method.

The rest of the paper is organized as follows. In
Section \ref{sec:Barron}, we describe Monte-Carlo sampling technique,
stratified sampling technique and Barron space. In Section
\ref{sec:appro}, we construct the finite neuron functions and prove
their approximation properties.  In Section \ref{sec:convergence}, we
propose the finite neuron method and provide the convergence
analysis. Finally, in Section \ref{sec:Summary}, we give some
summaries and discussions on the results in this paper.

Following  \cite{xu1992iterative}, we will use the notation
``$x\lesssim y$" to denote ``$x\le C y$" for some constant $C$
independent of crucial parameter such as mesh size.

\section{Preliminaries}
\label{sec:Barron}
In this section, for clarity of exposition, we present some standard materials from statistics about Monte Carlo sampling,  stratified sampling and their applications to analysis of asymptotic approximation properties of neural network functions.
\subsection{Monte-Carlo and stratified sampling techniques}
Let $\lambda\ge 0$ be a probability density function on a domain $G \subset
\mathbb R^D (D\ge 1)$ such that
\begin{equation}\label{density}
\int_{G}\lambda(\theta)d\theta=1.
\end{equation}
We define the expectation and variance as follows
\begin{equation}
\label{E}
\mathbb{E}g:=\int_{G}
g(\theta)\lambda(\theta)d\theta,\qquad
\mathbb{V} g: = \mathbb{E}((g - \mathbb{E} g)^2)= \mathbb{E}(g^2) - (\mathbb{E} g)^2.
\end{equation}
We note that 
$$\displaystyle \mathbb{V} g\le \max_{\theta, \theta'\in G}(g(\theta) - g(\theta'))^2.
$$
For any subset $G_i\subset G$, let
$$
\lambda(G_i)=\int_{G_i}\lambda(\theta)d\theta, \qquad \lambda_i(\theta) = {\lambda(\theta)\over \lambda(G_i)}.
$$
It holds that
$$
\mathbb{E}_Gg= \sum_{i=1}^M \lambda(G_i) \mathbb{E}_{G_i}g.
$$ 
For any function $h(\theta_1,\cdots, \theta_N) : G_1\times G_2\cdots G_N \mapsto \mathbb{R}$, define 
$$\mathbb{E}_{G_i}g=\int_{G_i}g(\theta)\lambda_i(\theta)d\theta$$
 and
\begin{equation}
\label{En}
\mathbb{E}_Nh:=\int_{G_1\times G_2\times\ldots\times G_N}
h(\theta_1,\cdots,\theta_N) \lambda_1(\theta_1) \lambda_2(\theta_2)\ldots \lambda_N(\theta_N)
d\theta_1d\theta_2\ldots d\theta_N.
\end{equation}
For the Monte Carlo method, let $G_i=G$ for all $1\le i\le n$, namely,
\begin{equation} 
\mathbb{E}_Nh:=\int_{G\times G\times\ldots\times G}
h(\theta_1,\cdots,\theta_N) \lambda(\theta_1) \lambda(\theta_2)\ldots \lambda(\theta_N)
d\theta_1d\theta_2\ldots d\theta_N.
\end{equation}

The following result is standard \cite{rubinstein2016simulation} and their proofs can be obtained by direct calculations.
\begin{lemma}   \label{MC}
For any $g\in L^\infty(G)$, we have
  \begin{equation}
  \begin{split}
    \mathbb{E}_N\Big(\mathbb{E}g-\frac1N\sum_{i=1}^N
    g(\omega_i)\Big)^2
    &=
    \left\{
             \begin{aligned}
    \frac{1}{N}\mathbb{V}(g)
   & \le\frac{1}{N} \sup_{\omega, \omega'\in G} |g(\omega) - g(\omega')|^2
    \\
\frac{1}{N}\Big(\mathbb{E}(g^2) - \big (\mathbb{E}(g)\big )^2\Big)
&\le\frac{1}{N} \mathbb E(g^2)\le \frac{1}{N}\|g\|^2_{L^\infty},
\end{aligned}
\right.
\end{split}
  \end{equation} 
\end{lemma}

\begin{proof}
First note that
 \begin{equation}
    \label{eqn}
    \begin{aligned}
\left(\mathbb{E} g-\frac1N\sum_{i=1}^Ng(\omega_i)\right)^2 
  & 
=\frac{1}{N^2} \left(\sum_{i=1}^N(\mathbb{E} g-g(\omega_i))\right)^2 
  \\
  &=\frac{1}{N^2} \sum_{i,j=1}^N(\mathbb{E} g-g(\omega_i))(\mathbb{E} g-g(\omega_j))
  \\
  &=\frac{I_1}{N^2} +\frac{I_2}{N^2}.
    \end{aligned}
  \end{equation}
with 
\begin{equation}
I_1= \sum_{i=1}^N(\mathbb{E} g-g(\omega_i))^2,\quad I_2=\sum_{i\neq  j}^N\left ((\mathbb{E}g)^2-\mathbb {E}(g)(g(\omega_i)+
g(\omega_j))+g(\omega_i)g(\omega_j))\right).
\end{equation}
Consider $I_1$, for any $i$,
 $$
 \mathbb{E}_N(\mathbb{E} g-g(\omega_i))^2
 =\mathbb{E}(\mathbb{E} g-g)^2 = \mathbb{V}(g).
 $$ 
Thus,
$$
 \mathbb{E}_N (I_1) = n\mathbb{V}(g).
$$
For $I_2$, note that
$$
\mathbb E_N g(\omega_i)=\mathbb E_N g(\omega_j) =\mathbb E(g)
$$
and, for $i\neq j$,
\begin{equation}\label{key}
\begin{aligned}
\mathbb {E}_N ( g(\omega_i)g(\omega_j)) &= 
\int_{G\times G\times\ldots\times G}
g(\omega_j) g(\omega_j) \lambda(\omega_1) \lambda(\omega_2)\ldots \lambda(\omega_N)
d\omega_1d\omega_2\cdots d\omega_N \\
&= \int_{G\times G} g(\omega_j) g(\omega_j) \lambda(\omega_1) 
\lambda(\omega_1) \lambda(\omega_2)
d\omega_1d\omega_2 \\
&= \mathbb {E}_N (
g(\omega_i))\mathbb {E}_n(g(\omega_j))
=[\mathbb E(g)]^2.
\end{aligned}
\end{equation}
Thus
 \begin{equation}
\mathbb{E}_N (I_2) = \mathbb{E}_N \left( \sum_{i\neq j}^N((\mathbb{E}g)^2-\mathbb
  E(g)(\mathbb E(g(\omega_i))+ \mathbb E(g(\omega_j)))+\mathbb E(g(\omega_i)g(\omega_j))) \right)=0.
  \end{equation}
Consequently, there exist the following two formulas for $\displaystyle \mathbb{E}_N\left(\mathbb{E} g-
      \frac1N\sum_{i=1}^Ng(\omega_i)\right)^2$:
 \begin{equation} 
 \mathbb{E}_N\left(\mathbb{E} g-
      \frac1N\sum_{i=1}^Ng(\omega_i)\right)^2 = \frac{1}{N^2}\mathbb{E}_N (I_1)
      =
           \left\{
             \begin{aligned}
            \frac{1}{N}\mathbb{E}\big ((\mathbb{E} g-g)^2\big )\\
            \frac{1}{N}(\mathbb{E}(g^2) - (\mathbb{E} g)^2).
            \end{aligned}
    \right.
  \end{equation}
Based on the first formula above, since
$$
|g(\omega) - \mathbb{E} g|=|\int_G \big (g(\omega) - g(\tilde \omega) \big )\lambda(\tilde \omega)d\tilde \omega|\le \sup_{\omega, \omega'\in G} |g(\omega) - g(\omega')|,
$$
it holds that
 \begin{equation} 
 \mathbb{E}_N\left(\mathbb{E} g-
      {1\over N }\sum_{i=1}^Ng(\omega_i)\right)^2 
            \le\frac{1}{N} \sup_{\omega, \omega\in G} |g(\omega) - g(\omega')|^2.
  \end{equation}
Due to the second formula above,  
 \begin{equation}
    \label{eqn}
 \mathbb{E}_N\left(\mathbb{E} g-
      \frac1N\sum_{i=1}^Ng(\omega_i)\right)^2  
            \le\frac{1}{N} \mathbb E(g^2)\le\frac{1}{N}\|g\|^2_{L^\infty}
  \end{equation}
which completes the proof.
\end{proof}

Stratified sampling  \cite{bickel1984asymptotic} gives a more refined version of the Monte Carlo method. 
\begin{lemma} \label{lem:stratified}
 For any
  nonoverlaping decomposition $G=G_1\cup G_2\cup \cdots \cup G_M$ and
  positive integer $n$, let $n_i=\lceil \lambda(G_i)n\rceil$ be the smallest integer larger than $\lambda(G_i)n$ and $\displaystyle N=\sum_{i=1}^M n_i$.  Let $\theta_{i,j}\in G_i (1\le j\le n_i)$ and 
\begin{equation}
g_n=\sum_{i=1}^M \lambda(G_i)g_{n_i}^i \quad \mbox{with}\quad g_{n_i}^i={1 \over n_i}\sum_{j=1}^{n_i} g(\theta_{i,j}).
\end{equation} 
It holds that  
\begin{equation}
\mathbb{E}_N(\mathbb{E}_G g -g_N)^2=\sum_{i=1}^M {\lambda^2(G_i) \over n_i} \mathbb{E}_{G_i}\big (g -\mathbb{E}_{G_i}g\big )^2
\le {1 \over n} \max_{1\le i\le M}\sup_{\theta, \theta'\in G_i}\big |g(\theta) - g(\theta')\big |^2.
\end{equation}
\end{lemma}
\begin{proof}
It follows from definition that
\begin{equation}
g(x, \theta) = \sum_{i=1}^M \lambda(G_i) g(x, \theta),\qquad \mathbb{E}_G g =\sum_{i=1}^M\lambda(G_i) \int_{G_i} g\lambda_i(\theta)d\theta = \sum_{i=1}^M \lambda(G_i) \mathbb{E}_{G_i}g.
\end{equation}
Thus, the difference $g - \mathbb{E}_G g$ is a linear combination of $g -\mathbb{E}_{G_i}g$ on each $G_i$ as follows
\begin{equation}
g - \mathbb{E}_G g = \sum_{i=1}^M \lambda(G_i) \big (g -\mathbb{E}_{G_i}g\big ).
\end{equation}
It follows from 
\begin{equation}
\mathbb{E}_G g-g_n=\sum_{i=1}^M \lambda(G_i) (\mathbb{E}_{G_i}g -  g^i_{ n_i})
\end{equation}
and \eqref{En} that
\begin{equation} 
\mathbb{E}_n(\mathbb{E}_G g -g_n)^2 =\sum_{i,j=1}^M \lambda(G_i) \lambda(G_j) \mathbb{E}_n \big ((\mathbb{E}_{G_i}g -  g^i_{ n_i})(\mathbb{E}_{G_j}g -  g^j_{ n_j})\big )
=\sum_{i,j=1}^M \lambda(G_i) \lambda(G_j) I_{ij}  
\end{equation}
with 
$$
I_{ij} = \mathbb{E}_n \big ((\mathbb{E}_{G_i}g -  g^i_{ n_i})(\mathbb{E}_{G_j}g -  g^j_{ n_j})\big ).
$$
By Lemma \ref{MC},
\begin{equation}
I_{ij} =  \mathbb{E}_n \big ((\mathbb{E}_{G_i}g -  g^i_{ n_i})^2\big )\delta_{ij}= {1\over n_i} \mathbb{E} \big ((\mathbb{E}_{G_i}g -  g)^2\big )\delta_{ij}.
\end{equation}
Thus,
\begin{equation} 
\mathbb{E}_n(\mathbb{E}_G g -g_n)^2 =\sum_{i=1}^M {\lambda^2(G_i) \over n_i} \mathbb{E}_{G_i}\big (g -\mathbb{E}_{G_i}g\big )^2, 
\end{equation}
which completes the proof.
\end{proof} 

Lemma \ref{MC} and Lemma \ref{lem:stratified} represent two simple identities and subsequent inequalities that can be verified by a direct calculation. Actually Lemma \ref{MC} is a special case of Lemma  \ref{lem:stratified} with $M=1$. Lemma \ref{MC} and Lemma \ref{lem:stratified} are the basis of Monte-Carlo sampling and  stratified sampling in statistics. In the presentation of the this paper, we choose not to use any concepts related to random samplings.

Given another domain $\Omega\subset \mathbb{R}^d$, we consider the case that $g(x, \theta)$ is a function of both $x\in \Omega$ and  $\theta\in G$. Given any function $\rho\in L^1(G)$, we consider
\begin{equation}\label{def:f}
u(x)=\int_{G}g(x,\theta)\rho(\theta)d\theta 
\end{equation} 
with $\|\rho\|_{L^1(G)}<\infty$. Let $\lambda(\theta)={\rho(\theta)\over \|\rho\|_{L^1(G)}} $. Thus,
\begin{equation}
u(x)=\|\rho\|_{L^1(G)}\int_{G}g(x,\theta)\lambda(\theta)d\theta
\end{equation}
with $\|\lambda(\theta)\|_{L^1(G)}=1$. 

We can apply the above two lemmas to the given function $u(x)$.

\begin{lemma}\label{lem:sample}
\textup{[Monte Carlo Sampling]}
	Consider $u(x)$ in \eqref{def:f} 
	with $0\le \rho(\theta)\in L^1(G)$. For any $N\ge 1$, there exist $\theta_i^*\in G$ such that
	$$
	\|u-u_N\|_{L^2(\Omega)}^2 
	\le\frac{1}{N}
	\int_G \|g(\cdot,\theta)\|_{L^2(\Omega)}^2\rho(\theta)d\theta = {\|\rho\|_{L^1(G)}\over N}\mathbb E (\|g(\cdot,\theta)\|_{L^2(\Omega)}^2)
	$$
	where  
	$
	\|g(\cdot,\theta)\|_{L^2(\Omega)}^2 = \int_{\Omega} [g(x,\theta)]^2 d\mu(x),
	$
	\begin{equation}\label{fndef} 
	u_N(x)=\frac{\|\rho\|_{L^1(G)}}{N}\sum_{i=1}^N g(x,\theta_i^*).
	\end{equation}

Similarly, if $g(\cdot, \theta)\in H^m(\Omega)$, for any $N\ge 1$, there exist $\theta_i^*\in G_i$ with $f_N$ given in \eqref{fndef} such that
	\begin{equation}\label{eq:hm}
	\|u-u_N\|_{H^m(\Omega)}^2 
	\le 
	\int_G  \|g(\cdot,\theta)\|_{H^m(\Omega)}^2\rho(\theta)d\theta
	=\frac{\|\rho\|_{L^1(G)}}{N} \mathbb E (\|g(\cdot,\theta)\|_{H^m(\Omega)}^2).
	\end{equation}
\end{lemma}

\begin{proof} 
Note that
\begin{equation}
\label{uv}
u(x) = \|\rho\|_{L^1(G)}\mathbb E (g).
\end{equation}
By Lemma \ref{MC},
$$
\mathbb {E}_n\left(\bigg(\mathbb E(g(x,\cdot))
-\frac{1}{N}\sum_{i=1}^N g(x,\theta_i))\bigg)^2
\right)\le {1\over N} \mathbb E (g^2).
$$
By taking integration w.r.t. $x$ on both sides, we get
$$
\mathbb {E}_n\left(h(\theta_1,\theta_2, \cdots, \theta_N)
\right)\le {1\over N} \mathbb {E} \Big(\int_{\Omega} g^2 d\mu(x)\Big),
$$
where 
$$
h(\theta_1,\theta_2, \cdots, \theta_N) =  \int_{\Omega} \bigg(\mathbb E(g(x,\cdot))
-\frac{1}{N}\sum_{i=1}^N g(x,\theta_i))\bigg)^2 d\mu(x).
$$
Sine $\mathbb {E}_N (1) = 1$ and $\mathbb {E}_N (h) \le {1\over N} \mathbb {E} \Big(\int_{\Omega} g^2 d\mu(x)\Big)$, there exist $\theta_i^* \in G$ such that
$$
h(\theta_1^*, \theta_2^*, \cdots, \theta_N^*) \le  {1\over N}  \int_{\Omega} \mathbb {E} (g^2) d\mu(x).
$$
This implies that
$$
	\mathbb{E}_n\|u-u_N\|_{L^2(\Omega)}^2 
	\le\frac{\|\rho\|_{L^1(G)}}{N}
	\int_G \|g(\cdot,\theta)\|_{L^2(\Omega)}^2\lambda(\theta)d\theta.
	$$ 
	The proof for \eqref{eq:hm} is similar to the above analysis for the $L^2$-error analysis, which completes the proof.
\end{proof}


\begin{lemma}\label{lem:stratifiedapprox}
\textup{[Stratified Sampling]}
	For $u(x)$ in \eqref{def:f} 
	with positive $\rho(\theta)\in L^1(G)$, given any positive integers $n$ and $M\le n$, for  any nonoverlaping decomposition $G=G_1\cup G_2\cup \cdots \cup G_M$, there exists$\{\theta_i^\ast\}_{i=1}^N$ with $n\le N \le 2n$ such that
	\begin{equation} 
	\|u - u_N\|_{L^2(\Omega)} \leq N^{-1/2}\|\rho\|_{L^1(G)}\max_{1\le j\le M}\sup_{\theta_{j},\theta_{j}'\in G_j} \| g(x,\theta_j) - g(x,\theta_j')\|_{L^2(\Omega)} 
	\end{equation}
	where 
	$$
	u_N(x)= {2\|\rho\|_{L^1(G)}\over N}\sum_{i=1}^N\beta_i g(x,\theta_i^\ast)
	$$ 
	and $\beta_i\in [0,1]$.
	\end{lemma}

\begin{proof}
Let $n_j=\lceil \lambda(G_j)n\rceil$ and $\theta_{i,j} \in G_j(1\leq i\leq n_j)$. Define $\displaystyle N=\sum_{j=1}^M n_j$ and
$$
u_N(x)=\|\rho\|_{L^1(G)}\sum_{i=1}^M \lambda(G_i)g_{n_i}^i \quad \mbox{with}\quad g_{n_i}^i={1 \over n_i}\sum_{j=1}^{n_i} g(\theta_{i,j}).
$$
Since
$
\displaystyle u(x)=\|\rho\|_{L^1(G)}\sum_{i=1}^M \lambda(G_i)\mathbb{E}_{G_i} g,
$
by Lemma \ref{MC},
\begin{equation}
\begin{split}
\mathbb{E}_n\|u- u_N \|_{L^2(\Omega)}^2=& \|\rho\|_{L^1(G)}^2
 \sum_{j=1}^{M}{\lambda^2(G_j) \over n_j}\mathbb{E}_{G_j}\|\mathbb{E}_{G_j} g -  g\|^2_{L^2(\Omega)}
\\
\le &\|\rho\|_{L^1(G)}^2 \sum_{j=1}^{M} {\lambda^2(G_j)\over n_j}\sup_{\theta_{j},\theta_{j}'\in G_j} \| g(x,\theta_j) - g(x,\theta_j')\|^2_{L^2(\Omega)}.
\end{split}
\end{equation} 
Since ${\lambda(G_j)\over n_j}\le {1\over n}$ and $\displaystyle \sum_{j=1}^M \lambda(G_j)=1$,
\begin{equation}
\mathbb{E}_n\|u - u_N \|_{L^2(\Omega)}^2\leq n^{-1}\|\rho\|_{L^1(G)}^2\max_{1\le j\le M}\sup_{\theta_{j},\theta_{j}'\in G_j} \| g(x,\theta_j) - g(x,\theta_j')\|^2_{L^2(\Omega)}.
\end{equation}
There exist $\{\theta_{i,j}^\ast\}$ such that $\theta_{i,j}^\ast\in G_i$ and 
\begin{equation}
\| u-u_N \|_{L^2(\Omega)}^2\leq n^{-1} \|\rho\|_{L^1(G)}^2\max_{1\le j\le M}\sup_{\theta_{j},\theta_{j}'\in G_j} \| g(x,\theta_j) - g(x,\theta_j')\|^2_{L^2(\Omega)}.
\end{equation}
Note that $n\le N\le n+M\le 2n$,
$$
u_N(x) =  {2\|\rho\|_{L^1(G)}\over N}\sum_{j=1}^{M}\frac{N\lambda(G_j)}{2n_j}\sum_{i=1}^{n_j} g(x,\theta_{i,j}^\ast)
 =  {2\|\rho\|_{L^1(G)}\over N}\sum_{j=1}^{M}\beta_{i,j}\sum_{i=1}^{n_j} g(x,\theta_{i,j}^\ast)
$$ 
with
\begin{equation}
\beta_{i,j}= \frac{N\lambda(G_j)}{2n_j}\le \frac{2\lambda(G_j)n}{2\lambda(G_j)n}\le 1,
\end{equation}
which completes the proof. 
\end{proof}

\subsection{Barron spectral space}	
Let us use a simple example to motivate the Barron space. 
Consider the Fourier transform of a real function $u: \mathbb
R^d\mapsto \mathbb R$
\begin{equation}
  \label{eq:1}
\hat u (\omega)=(2\pi)^{-d/2}\int_{\mathbb R^d}e^{-i\omega\cdot x}u(x)dx.
\end{equation}
This gives the following integral representation of $u$ in terms of the cosine function
\begin{equation}
 \label{eq:reint}
u(x)=Re\int_{\mathbb{R}^d} e^{i\omega\cdot x} \hat u(\omega)d\omega
= \int_{\mathbb{R}^d}\cos (\omega\cdot x + b(\omega)) |\hat u(\omega)|d\omega,
\end{equation}
where $ \hat u(\omega)= e^{ib(\omega)}|\hat u(\omega)|$. Let 
\begin{equation}
 \label{eq:2}
g(x, \omega) =\cos (\omega\cdot x +b(\omega))\quad \mbox{ and }\quad 
\rho(\omega)= |\hat u(\omega)| . 
\end{equation}
Thus, 
\begin{equation}
\label{int-rep}
u(x)= \int_{\mathbb{R}^d}g(x,\omega) \rho(\omega)d\omega,   
\end{equation}
If
$$
\int_{\mathbb R^d} |\hat u(\omega)|d\omega <\infty,
$$
then $\|\rho\|_{L^1}<\infty$. By applying the Lemma \ref{lem:sample},
there exist $\omega_i\in \mathbb R^d$
such that
\begin{equation}
  \label{eq:3}
\|u-u_N\|_{0,\Omega}\le N^{-1/2}\|\hat u\|_{L^1(\mathbb R^d)}.  
\end{equation}
where
\begin{equation}\label{cosfn}
u_N(x) = {\|\hat u\|_{L^1(\mathbb R^d)}\over N} \sum_{i=1}^N \cos (\omega_i\cdot x + b(\omega_i)) 
\end{equation}
More generally, we consider the approximation property in $H^m$-norm.
By \eqref{eq:reint},
\begin{equation} 
\partial^\alpha u(x)= \int_{\mathbb{R}^d} \cos^{|\alpha|} (\omega\cdot x +b(\omega))\omega^\alpha |\hat u(\omega)|d\omega,  \quad \forall\ |\alpha|\le m.
\end{equation}
For any positive integer $m$, let 
\begin{equation} \label{eq:gm}
g_m(x,\omega)= {\cos (\omega\cdot x +b(\omega))\over  (1+ \|\omega\|)^m}\quad \mbox{and}\quad \rho_m(\omega)= (1+ \|\omega\|)^m|\hat u(\omega) |,
\end{equation}
where
$$
\| \rho_m\|_{L^1(\mathbb R^d)}=\int_{\mathbb R^d} (1+ \|\omega\|)^m|\hat u(\omega) | d\omega<\infty.
$$
Then, $\displaystyle u(x)=\int_{\mathbb R^d} g_m(x,\omega)\rho_m d\omega = \| \rho_m\|_{L^1(\mathbb R^d)}\mathbb{E}g_m(x,\omega)$. Define
\begin{equation}
u_N(x) = {\|\rho_m\|_{L^1(\mathbb R^d)}\over N} \sum_{i=1}^N g_m(x,\omega_i)
= {\|\rho_m\|_{L^1(\mathbb R^d)}\over N} \sum_{i=1}^N {\cos(\omega_i\cdot x + b(\omega_i))\over (1+\|\omega_i\|)^m}.
\end{equation}
It holds that 
$$
\partial^\alpha (u(x) - u_N(x))={\|\rho_m\|_{L^1(\mathbb R^d)}\over N}\sum_{i=1}^N \mathbb{E} \partial^\alpha (g_m(x,\omega) -  g_m(x,\omega_i)).
$$
By Lemma \ref{MC},
\begin{align}
\mathbb{E}_N \sum_{|\alpha|\le m}\|\partial^\alpha (u(x) - u_N(x))\|_{0, \Omega}^2 
&\le 
\|\rho_m \|_{L^1(\mathbb R^d)}^2\mathbb{E}_N \sum_{|\alpha|\le m}\frac{1}{N^2}\sum_{i=1}^N \left (\mathbb{E} \partial^\alpha (g_m(x,\omega) -  g_m(x,\omega_i))\right )^2
\\
&\le 
{\|\rho_m \|_{L^1(\mathbb R^d)}^2\over N}  \sum_{|\alpha|\le m}\mathbb{E} \left (\partial^\alpha g_m(x,\omega)\right)^2
\end{align}
Note that the definitions of $g_m$ and $\rho_m$ in \eqref{eq:gm} guarantee that 
$$
|\partial^\alpha g_m(x,\omega)|\le 1.
$$
Thus,
$$
\mathbb{E}_N \sum_{|\alpha|\le m}\|\partial^\alpha (u(x) - u_N(x))\|_{0, \Omega}^2 \lesssim {\|\rho_m \|_{L^1(\mathbb R^d)}^2\over N}.
$$
This implies that there exist
$\omega_i\in \mathbb R^d$ such that
\begin{equation}
\label{cosHm}
\|u-u_N\|_{H^m(\Omega)}\lesssim N^{-1/2}\int_{\mathbb{R}^d} (1+ \|\omega\|)^m|\hat u(\omega) | d\omega.  
\end{equation}

Given $v\in L^2(\Omega)$,   consider all the possible extension $v_E:
\mathbb{R}^d \mapsto \mathbb{R}$ with $v_E |_{\Omega} = v$ and define
the Barron spectral norm for any $s\ge 1$:
	\begin{equation}\label{barron-norm}
	\|v\|_{B^{s}(\Omega)} = \inf_{v_E |_{\Omega} = v} \int_{\mathbb{R}^d}(1+\|\omega\|)^s|\hat{v}_E(\omega)|d\omega
	\end{equation}
and Barron spectral space
\begin{equation}
  \label{Barron}
	B^{s}(\Omega) = \{v\in L^2(\Omega): \|v\|_{B^{s}(\Omega)}<\infty\}.  
\end{equation}

In summary, we have 
\begin{equation} 
\|u-u_N\|_{H^m(\Omega)}\lesssim N^{-1/2}  \|u\|_{B^{m}(\Omega)}.
\end{equation}

The estimate of \eqref{eq:3}, first obtained in \cite{jones1992simple} using a slightly different technique, appears to be the first asymptotic error estimate for the artificial neural network. \cite{barron1993universal} extended Jones's estimate \eqref{eq:3} to sigmoidal type of activation function in place of cosine.

The above short discussions reflect the core idea in the analysis of
approximation property of artificial neural networks.  Namely,
represent $f$ as an expectation of some probability distribution as in
\eqref{int-rep} and then a simple application of Monte-Carlo sampling
then leads to error estimate like \eqref{cosHm} for a special neural
network function given by \eqref{cosfn} using $\sigma=\cos$ an
activation function.   For a more general activation function $\sigma$, we just need to
derive a corresponding representation like \eqref{int-rep} with $g$ in
terms of $\sigma$. %
Quantitative estimates on the order of approximation are obtained for
sigmoidal activation functions are obtained in \cite{barron1993universal} and for
periodic activation functions in
\cite{mhaskar1995degree,mhaskar1994dimension}. 
Error estimates in Sobolev norms for general
activation functions can be found in \cite{hornik1994degree}.  A
review of a variety of known results, especially for networks with one
hidden layer, can be found in \cite{pinkus1999approximation}.  More
recently, these results have been improved by a factor of $n^{1/d}$ in
\cite{klusowski2016uniform} using the idea of stratified sampling,
based in part on the techniques in \cite{makovoz1996random}.
\cite{siegel2020approximations} provides an analysis for general
activation functions under very weak assumptions which applies to
essentially all activation functions used in practice.  In
\cite{ma2019barron,ma2019priori,wojtowytsch2020representation},
a more refined definition of the Barron norm is introduced to give sharper
approximation error bounds of neural networks.

The following lemma shows some relationship between Sobolev norm and
the Barron spectral norm.

\begin{lemma}\label{smoothness-lemma}
 Let $m \geq 0$ be an integer and $\Omega\subset \mathbb{R}^d$ a bounded domain. Then for any Schwartz function $v$, we have
 \begin{equation}\label{embend}
 \|v\|_{H^m(\Omega)} \lesssim \|v\|_{{B}^m(\Omega)} \lesssim  \|v\|_{H^{m + {d\over 2}+\epsilon}(\Omega)}.
 \end{equation}
\end{lemma}
\begin{proof} 
The first inequality in \eqref{embend} and its proof can be found in 
  \cite{siegel2020approximations}.  A version of the second inequality
  in \eqref{embend} and its proof 
can be found in \cite{barron1993universal}. Below is a proof, by 
definition and Cauchy-Schwarz
  inequality, 
\begin{align}
\|v\|_{B^m(\Omega)} =& \inf_{v_E |_{\Omega} = v} \left(\int_{\mathbb{R}^d}(1+\|\omega\|)^m|\hat{v}_E(\omega)|d\omega \right)^2
\\
\le &  \int_{\mathbb{R}^d}(1+\|\omega\|)^{-d - 2\epsilon}d\omega  \inf_{v_E |_{\Omega} = v} 
\int_{\mathbb{R}^d}(1+\|\omega\|)^{d + 2m + 2\epsilon} |\hat{v}_E(\omega)|^2d\omega  
\\
\lesssim &  \inf_{v_E |_{\Omega} = v} 
\int_{\mathbb{R}^d}(1+\|\omega\|)^{d + 2m + 2\epsilon} |\hat{v}_E(\omega)|^2d\omega  
\lesssim \|v\|_{H^{m + {d\over 2}+\epsilon}(\Omega)}.
\end{align}
\end{proof}

\section{Finite neuron functions and approximation properties}\label{sec:appro}
As mentioned before, for $m=1$,  the finite element for
\eqref{m-BD} can be given by piecewise linear function in
any dimension $d\ge1$.    As shown in \cite{he2020relu},  the linear
finite element function can be represented by deep neural network with
ReLU as activation functions.   Here 
\begin{equation}
  \label{relu}
\relu(x)=x_+=\max(0,x).  
\end{equation}
In this paper, we will consider the power of ReLU as activation functions
\begin{equation}
  \label{relup}
\relu^k(x)=[x_+]^k=[\max(0,x)]^k. 
\end{equation}
We will use a short-hand notation that $x_+^k= [x_+]^k$ in the rest of the paper.

We consider the following neuron network functional class 
with one hidden layer:
\begin{equation}
\label{VkN}
V_N^k=\left\{\sum_{i=1}^Na_i(w_ix+b_i)_+^k, a_i, b_i\in\mathbb R^1, w_i\in \mathbb R^{1\times d}\right\}.
\end{equation}
We note that $V_N^k$ is not a linear vector space.  The definition of
neural network function class such as \eqref{VkN} can be traced back
in \cite{mcculloch1943logical} and its early mathematical analysis can be found in \cite{hornik1989multilayer,cybenko1989approximation,funahashi1989approximate}.

The functions in $V^k_N$ as defined in \eqref{VkN} will be known as finite neuron functions in this paper.
\begin{lemma}
  For any $k\ge 1$, $V_N^k$ consists of functions that are piecewise
  polynomials of degree $k$ with respect to a grid whose boundaries are
  given by intersection of the following hyperplanes
$$
w_ix + b_i=0,\quad 1\le i\le N.
$$
see Fig \ref{fig:1} - \ref{fig:3}. Furthermore, if $k\ge m$,
$$
V_N^k(\Omega)\subset H^k(\Omega)\subset H^m(\Omega).
$$
\end{lemma}
\begin{figure}[!ht]
\begin{center}
\includegraphics[width=.3\textwidth]{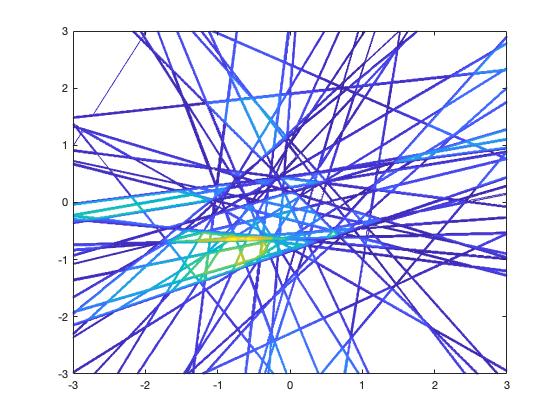}    
\end{center}
\caption{Hyperplanes with $\ell=1$}
\label{fig:1}
\end{figure} 		



The main goal of this section is to prove that the 
following type of error estimate holds, for some $\delta\ge 0$,
\begin{equation}\label{VNerror}
\inf_{v_N\in V_N^k}\|u- v_N \|_{H^m(\Omega)} \lesssim 
N^{-{1\over 2}-\delta}.
\end{equation} 
We will use two different approaches to establish \eqref{VNerror}.
The first approach, presented in \S\ref{sec:Bsplines}, mainly follows
\cite{hornik1994degree} and \cite{siegel2020approximations} that gives
error estimates for a general class of activation functions.  The
second approach, presented in \S\ref{sec:error2}, follows
\cite{klusowski2016uniform} that gives error estimates specifically
for ReLU activation function.

We assume that $\Omega\subset\mathbb R^d$ is a given bounded domain.
Thus,
\begin{equation}
  \label{T}
T=\max_{x\in \bar{\Omega}} \|x\|<\infty.  
\end{equation}

\subsection{B-spline as activation functions}\label{sec:Bsplines}
The activation function ${\rm [ ReLU]}^k$ \eqref{relup} are related to
cardinal B-Splines.  A cardinal B-Spline of degree $k\ge 0$
denoted by $b^k$, is defined by convolution as
\begin{equation}
	b^k(x)=(b^{k-1}*b^0)(x)=\int_\mathbb{R}b^{k-1}(x-t)b^0(t)dt,
\end{equation}
where 
\begin{equation}
b^0(x)=\left\{
		     \begin{array}{lr}
		    1 & x\in[0,1),\\
		    0 & \hbox{otherwise}.
		     \end{array}
	\right.
\end{equation}
\begin{figure}
\begin{center}
\includegraphics[width=0.5\textwidth]{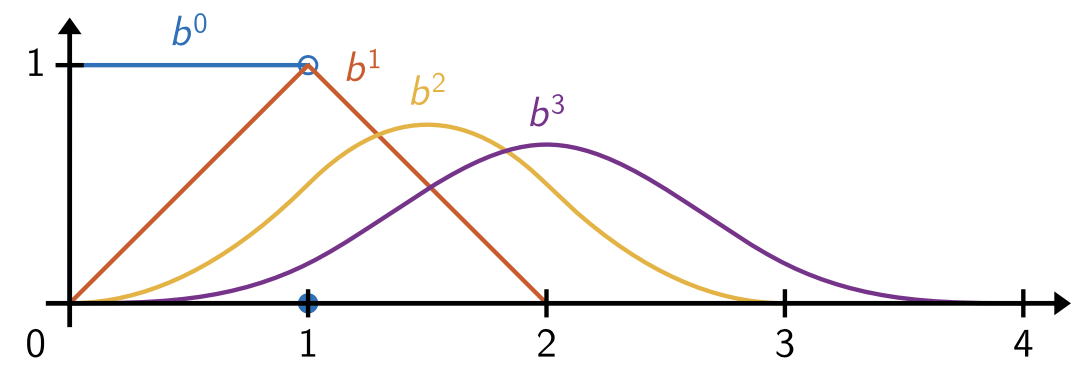}   
\caption{Plots of some B-spline basis}
\label{bk}
\end{center}
\end{figure}
More explicitly, see \cite{de1971subroutine}, for any
$x\in[0,k+1]$ and $k\geq 1$, we have 
	\begin{equation}
	b^k(x)=\frac{x}{k}b^{k-1}(x)+\frac{k+1-x}{k}b^{k-1}(x-1),
	\end{equation}
or
	\begin{equation}\label{splinetorelu}
	b^k(x)=(k+1)\sum_{i=0}^{k+1} w_i(i-x)_+^k \hbox{~and~} w_i={\displaystyle\prod_{j=0,j\neq i}^{k+1}} \frac{1}{i-j}.
	\end{equation}
We note that all $b^k$ are locally supported and see Fig.~\ref{bk} for their plots. 

For an uniform grid with mesh size $h=\frac{1}{n+1}$, we define
	\begin{equation}
	b^k_{j,h}(x)=b^{k}(\frac{x}{h}-j).
	\end{equation}
Then the cardinal B-Spline series of degree $k$ on the uniform grid is 
\begin{equation}\label{Skn}
S_n^k=\Big\{v(x)=\sum_{j=-k}^{n}	c_jb^k_{j,h}(x)\Big\}.
\end{equation}
\begin{lemma} For $V_N^k$ and $S_n^k$ defined by \eqref{VkN} and
  \eqref{Skn}, we have
\begin{equation}
    \label{SV}
S_n^k\subset V_{n+k+1}^k.    
  \end{equation}
As a result, for any bounded domain $\Omega\subset \mathbb R^1$, we have
\begin{equation}
  \label{SVerror}
\inf_{v\in V_{N}^k}\|u-v\|_{m,\Omega} 
\le \inf_{v\in S_{N-k-1}^k}\|u-v\|_{m,\Omega} \lesssim N^{m-(k+1)} \|u\|_{k+1,\Omega}
\end{equation}
\end{lemma}

Given an activation function $\sigma$, consider its Fourier transformation:
\begin{equation}
  \label{Fsigma}
\hat \sigma(a) = \frac{1}{2\pi}\int_{\mathbb{R}} \sigma(t)e^{-iat}dt. 
\end{equation}
For any $a\neq 0$ with $\hat \sigma(a)\neq 0$, by making a change of variables 
$t = a^{-1}\omega\cdot x + b$ and $dt = db$, we have
 \begin{equation}
 \begin{aligned}
\hat{\sigma}(a)&=\frac{1}{2\pi}\int_{\mathbb{R}}\sigma(a^{-1}\omega\cdot x+b)e^{-ia( a^{-1}\omega\cdot x+b)}db = e^{-i\omega \cdot x}\frac{1}{2\pi}\int_{\mathbb{R}}\sigma(a^{-1} \omega\cdot x+b)e^{-iab}db.
 \end{aligned}
 \end{equation}
This implies that
\begin{equation}\label{FourierExp}
 e^{i\omega \cdot x} = \frac{1}{2\pi\hat{\sigma}(a)}\int_{\mathbb{R}}\sigma(a^{-1}\omega\cdot x+b)e^{-iab}db.
\end{equation}
We write $ \hat{u}(\omega) = e^{-i\theta(\omega)} | \hat{u}(\omega)|$
and then obtain the following integral represntation:
\begin{equation}\label{integral_representation01}
u(x) = \int_{\mathbb{R}^d} e^{i\omega\cdot x}\hat{u}(\omega)d\omega = 
\int_{\mathbb{R}^d}\int_\mathbb{R}\frac{1}{2\pi\hat{\sigma}(a)}
\sigma\left(a^{-1} \omega\cdot x+b\right)|\hat{u}(\omega)|e^{-i(ab+\theta(\omega))}dbd\omega
\end{equation}
Now we consider activation function $\sigma(x)=b^k(x)$ and $\hat \sigma$ be the
Fourier transform of $\sigma(x)$. Note that, by \eqref{bk}, 
\begin{align}\label{splineFourier}
\hat{\sigma}(a)=\left({1-e^{-ia}\over ia}\right)^{k+1}=\left({2\over a}\sin {a\over 2}\right)^{k+1}e^{-{ia(k+1)\over 2}}.
\end{align}

We first take $a=\pi$ in \eqref{splineFourier}. Thus,
\begin{equation}
  \label{pi1}
\hat\sigma(\pi)=
\left({2\over \pi}\right)^{k+1}e^{-{i\pi (k+1)\over 2}}.
\end{equation}
Combining \eqref{integral_representation01} and \eqref{pi1}, we obtain
that
\begin{equation}
  \label{splinerep0}
u(x) = 
\frac{1}{4}\left({\pi\over 2}\right)^{k}\int_{\mathbb{R}^d}\int_\mathbb{R}
\sigma\left(\pi^{-1}\omega\cdot x+b\right)|\hat{u}(\omega)|e^{-i(\pi b + {\pi (k+1)\over 2}+\theta(\omega))}dbd\omega
\end{equation}
An application of the Monte Carlo method in Lemma \ref{MC} to the integral representation \eqref{splinerep0} leads to the following estimate.
\begin{theorem}\label{splinestratify}
For any $0\le m\le k$, there exist $\omega_i\in \mathbb{R}^d$, $b_i\in \mathbb{R}$ such that
\begin{equation}
\left \|u - u_{N}\right\|_{H^{m}(\Omega)}\lesssim  N^{-{1\over 2}} \|u\|_{{B}^{m+1}(\Omega)}
\end{equation}
with
\begin{equation}
u_{N}(x)=\sum_{i=1}^{N} \beta_i b^k\left(\pi^{-1} \omega_i\cdot x+b_i\right).
\end{equation} 
\end{theorem}

Based on the integral representation \eqref{splinerep0}, a stratified analysis similar to the one in \cite{siegel2020approximations} leads to the following result.
\begin{theorem}
For any $0\le m\le k$ and positive $\epsilon$, there exist there exist $\omega_i\in \mathbb{R}^d$, $b_i\in \mathbb{R}$ such that
\begin{equation}\label{straunbdd}
\left\|u - u_{N}\right\|_{H^{m}(\Omega)}\le  N^{-{1\over 2}-{\epsilon \over (d+1)(2+\epsilon)}} \|u\|_{{B}^{m+1+\epsilon}(\Omega)}
\end{equation}
with
\begin{equation}
u_{N}(x)=\sum_{i=1}^{N} \beta_i b^k\left(\bar \omega_i\cdot x+b_i\right) .
\end{equation} 
\end{theorem}
Next, we try to improve the estimate \eqref{straunbdd}. Again, we will use \eqref{integral_representation01}. Let $\displaystyle a_\omega=4\pi\lceil {\|\omega\|\over 4\pi}\rceil + \pi$ in \eqref{splineFourier} and $\displaystyle \bar\omega ={\omega\over a_\omega}$. We have
 \begin{equation}
\hat{\sigma}(a_\omega)=\left({2\over a_\omega}\right)^{k+1}, \quad \|\omega\| + \pi\le a_\omega\le \|\omega\|+5\pi,\quad \|\bar\omega\|\le 1,
 \end{equation}
which, together with \eqref{integral_representation01}, indicates that
 \begin{equation}\label{integral_representation}
 \begin{split}
  u(x) =  \int_{\mathbb{R}^d}\int_\mathbb{R}\frac{1}{2\pi}
  \sigma\left(\bar \omega\cdot x+b\right)\left({a_\omega\over 2}\right)^{k+1}\hat{u}(\omega)e^{-ia_\omega(b+{k+1\over 2})}dbd\omega.
\end{split}
 \end{equation}

\begin{theorem}
There exist $\|\bar \omega_i\|\le 1$, $|b_i|\le T + k+1$ such that
\begin{equation}\label{d+1}
\left\|u - u_{N}\right\|_{H^{m}(\Omega)}\lesssim  N^{-{1\over 2}-{1\over d+1}} \|u\|_{{B}^{k+1}(\Omega)}
\end{equation}
with
\begin{equation}
u_{N}(x)=\sum_{i=1}^{N} \beta_i b^k\left(\bar \omega_i\cdot x+b_i\right) .
\end{equation} 
\end{theorem}
\begin{proof}
We write \eqref{integral_representation} as follows
$$
\displaystyle u(x)= \int_{\mathbb{R}^d}\int_\mathbb{R}
g(x, b, \omega)\rho(b,\omega) dbd\omega
$$
with 
$$ 
\hat{u}(\omega) = e^{-i\theta(\omega)} | \hat{u}(\omega)|,\quad \tilde \theta(\omega)=\theta(\omega) + a_\omega(b+{k+1\over 2})
$$ and 
\begin{equation}\label{eq:g}
g(x, b, \omega) = \sigma\left({\bar \omega}\cdot x+b\right)sgn(\cos \tilde\theta(\omega)),
\end{equation}
\begin{equation}\label{eq:rho}
\rho(b,\omega) = \frac{1}{(2\pi)^d}\left({a_\omega\over 2}\right)^{k+1}| \hat{u}(\omega)||\cos \tilde\theta(\omega)|.
\end{equation} 
Note that
\begin{equation}
\|\bar\omega\|\le 1, \quad |b|\le T+k+1.
\end{equation} 
Let 
$$
G=\{(\omega, b): \omega\in \mathbb{R}^d,\ |b|\le T+k+1\}, \ \tilde G=\{(\bar \omega, b): \|\bar \omega\| \le 1,\ |b|\le T+k+1\}.
$$
For any positive integer $n$, divide $\tilde  G$ into $\tilde  M(\tilde  M\le {n\over 2})$   nonoverlapping subdomains, say 
$\tilde  G=\tilde  G_1\cup \tilde  G_2\cup \cdots \cup\tilde  G_{\tilde M}$, such that
\begin{equation}
|b-b'|\lesssim n^{-{1\over d+1}},\quad |\bar\omega - \bar\omega'|\lesssim  n^{-{1\over d+1}}, \quad (\bar\omega, b),\ (\bar\omega', b')\in \tilde G_i,\ 1\le i\le \tilde M.
\end{equation} 
Define $M=2\tilde  M$ and for $1\le i\le \tilde M$,
$$
G_i = \{(\omega, b): (\bar \omega, b)\in \tilde  G_i, \ \cos \tilde\theta(\omega)\ge 0\},\ 
G_{\tilde  M+i} = \{(\omega, b): (\bar \omega, b)\in \tilde  G_i, \ \cos \tilde\theta(\omega)\le 0\}.
$$
Thus, $G=G_1\cup G_2\cup \cdots \cup G_M$ with $\tilde  G_i\cap \tilde G_j=\varnothing$ if $i\neq j$, and 
\begin{equation}
|b-b'|\lesssim n^{-{1\over d+1}},\quad |\bar\omega - \bar\omega'|\lesssim  n^{-{1\over d+1}}, \quad sgn(\cos \tilde\theta(\omega))=sgn(\cos\tilde \theta(\omega')).
\end{equation} 
Let $n_i=\lceil \lambda(G_i)n\rceil$, $N=\displaystyle \sum_{i=1}^M n_i$ and
\begin{equation}
u_{N}(x)=\|\rho\|_{L^1(G)}\sum_{i=1}^{M} \frac{\lambda(G_{i})}{n_{i}} \sum_{j=1}^{n_{i}}g(x,\theta_{i,j}).
\end{equation}
It holds that
\begin{equation}\label{eq:sum}
\begin{split}
\mathbb{E}\left(\left\|u - u_{N}\right\|_{H^{m}(\Omega)}^{2}\right)\le&
\|\rho\|_{L^1(G)}\sum_{i=1}^{M}  \frac{\lambda^2(G_i)}{n_{i}}  \sup_{\theta_{i},\theta_{i}'\in G_i} \| g(x,\theta_i) - g(x,\theta_i')\|^2_{H^m(\Omega)}
 \end{split}
\end{equation}
with $\theta=(b, \omega)$. 
For any $(b, \omega)\in G_i$, $1\le i\le M$, if $k\ge m+1$,
\begin{equation}
|g(x,\theta) - g(x,\theta')| \lesssim |b-b'| + |\omega - \omega'| \lesssim   n^{-{1\over d+1}}
\end{equation}
Thus,
\begin{equation}
 \sum_{i=1}^{M}  \frac{\lambda^2(G_i)}{n_{i}}  \sup_{\theta_{i},\theta_{i}'\in G_i} \| g(x,\theta_i) - g(x,\theta_i')\|^2_{H^m(\Omega)}  
 \lesssim  n^{-{2\over d+1}} |\Omega|.
\end{equation}
Thus,
\begin{equation}\label{eq:}
\begin{split}
\mathbb{E}\left(\left\|u - u_{N}\right\|_{H^{m}(\Omega)}^{2}\right)\lesssim&   n^{-1-{2\over d+1}} |\Omega|\|\rho\|_{L^1(G)}.
 \end{split}
\end{equation}
Since $a\le \|\omega\|+5\pi$, 
$$
\|\rho\|_{L^1(G)}\lesssim \int_G (\|\omega\| + 1)^{k+1}|\hat u(\omega)|d\omega db \lesssim \|u\|_{B^{k+1}(\Omega)}.
$$
Note that $n\le N\le 2n$. Thus, there exist $\omega_i\in \mathbb{R}^d$, $\beta_i$, $b_i\in \mathbb{R}$ such that
\begin{equation}
\left\|u - u_{N}\right\|_{H^{m}(\Omega)}\lesssim  N^{-{1\over 2}-{1\over d+1}} \|u\|_{B^{k+1}(\Omega)},
\end{equation}
which completes the proof.
\end{proof}

The above analysis can also be applied to more general activation functions with compact support. 
\begin{theorem}
Suppose that $\sigma\in W^{m+1,\infty}(\mathbb{R})$ that has a compact
support. If for any $a>0$, there exists $\tilde a>0$ such that
\begin{equation}
\tilde a\gtrsim a,\quad  |\hat\sigma(\tilde a)|\gtrsim a^{-\ell},
\end{equation}
then, there exist $\omega_i\in \mathbb{R}^d$ and $b_i\in \mathbb{R}$ such that
\begin{equation}
\left\|u - u_{N}\right\|_{H^{m}(\Omega)}\lesssim  N^{-{1\over 2}-{1\over d+1}} \|u\|_{B^{\ell}(\Omega)},
\end{equation}
where
\begin{equation}
u_{N}(x)=\sum_{i=1}^{N} \beta_i \sigma\left(\bar \omega_i\cdot x+b_i\right) .
\end{equation} 
\end{theorem}

\subsection{[ReLU]$^k$ as activation functions}\label{sec:error2}
Rather than using general Fourier transform as in \eqref{FourierExp} to represent
$e^{i\omega\cdot x}$ in terms of $\sigma(\omega\cdot x+b)$, 
\cite{klusowski2016uniform} gave a different method to represent
$e^{i\omega\cdot x}$ in terms of $(\omega\cdot x+b)_+^k$  for $k=1$
and $2$.   The following lemma gives a generalization of this
representation for all $k\ge 0$. 
\begin{lemma}\label{lm:talorcomplex}
For any $k\ge0$ and $x\in \Omega$,
\begin{equation}  
e^{i\omega\cdot x} =\sum_{j=0}^k{(i\omega\cdot x)^{j}\over j!} 
+
{i^{k+1}\over k!} \|\omega\|^{k+1}\int_{0}^T\left[(\bar \omega\cdot x - t)_+^ke^{i\|\omega\|t}
+(-1)^{k-1}(-\bar \omega\cdot x - t)_+^ke^{-i\|\omega\|t} \right]dt.
\end{equation} 
\end{lemma}
\begin{proof}  
For $|z|\leq c$, by the Taylor expansion with integral remainder,
\begin{equation} 
e^{iz} = \sum_{j=0}^k {(iz)^j\over j!} + {i^{k+1}\over k!} \int_0^z e^{iu}(z-u)^kdu.
\end{equation}
Note that 
$$
(z-u)^k=(z-u)^k_+ - (u-z)^k_+.
$$
It follows that
\begin{equation}
\begin{split}
\int_{0}^z (z-u)^ke^{iu} du=&\int_{0}^z (z-u)_+^ke^{iu} du + \int_{0}^z (-1)^k(u-z)_+^ke^{iu} du
\\
=&\int_{0}^z (z-u)_+^ke^{iu} du + \int_{0}^{-z} (-1)^{k-1}(-u-z)_+^ke^{-iu} du
\\
=&\int_{0}^c (z-u)_+^ke^{iu} du + (-1)^{k-1}(-u-z)_+^ke^{-iu} du.
\end{split}
\end{equation}
Thus,
\begin{equation}  
e^{iz} - \sum_{j=0}^k{(iz)^{j}\over j!} 
= 
{i^{k+1}\over k!}\int_{0}^c\left[(z - u)_+^ke^{iu} + (-1)^{k-1}(-z - u)_+^ke^{-iu} \right]du.
\end{equation}  
Let 
\begin{equation}\label{baromega}
z=\omega\cdot x,\quad u=\|\omega\|t,\quad \bar \omega={\omega\over \|\omega\|}.
\end{equation}
Since $\|x\| \le T$ and $|\bar \omega \cdot x|\le T$, we obtain
\begin{equation}  
e^{i\omega\cdot x} - \sum_{j=0}^k{(i\omega\cdot x)^{j}\over j!} 
= 
{i^{k+1}\over k!} \|\omega\|^{k+1}\int_{0}^T\left[(\bar \omega\cdot x - t)_+^ke^{i\|\omega\|t}
+(-1)^{k-1}(-\bar \omega\cdot x - t)_+^ke^{-i\|\omega\|t} \right]dt,
\end{equation} 
which completes the proof.
\end{proof}

Since $
u(x) = {1\over (2\pi)^d}\int_{\mathbb{R}^d} e^{i\omega\cdot x}\hat{u}(\omega)d\omega
$
and 
$
 \partial^\alpha u(x)=\int_{\mathbb{R}^d} i^{|\alpha|}\omega^\alpha e^{i\omega\cdot x}\hat{u}(\omega)d\omega,
$
\begin{eqnarray}
 \partial^\alpha u(0)x^\alpha=\int_{\mathbb{R}^d} i^{|\alpha|}\omega^\alpha x^\alpha\hat{u}(\omega)d\omega.
\end{eqnarray} 
Note that $\displaystyle (\omega\cdot x)^j=\sum_{|\alpha|=j}{j!\over \alpha !}\omega^\alpha x^\alpha $. It follows that
\begin{equation}
\sum_{|\alpha|=j}{1\over \alpha!} \partial^\alpha u(0) x^\alpha=i^j\sum_{|\alpha|=j}{1\over \alpha!} \int_{\mathbb{R}^d} \omega^\alpha x^\alpha\hat{u}(\omega)d\omega
={1\over j!}  \int_{\mathbb{R}^d} (i\omega\cdot x)^j \hat{u}(\omega)d\omega.
\end{equation} 
Let $\hat{u}(\omega)=|\hat{u}(\omega)|e^{ib(\omega)}$. Then, $e^{i\|\omega\|t}\hat{u}(\omega) = |\hat{u}(\omega)|e^{i(\|\omega\|t + b(\omega))}$.
By Lemma \ref{lm:talorcomplex},
\begin{equation}\label{eq:fftaylor}
\begin{split}
&u(x) - \sum_{|\alpha|\le k}{1\over \alpha!} \partial^\alpha u(0) x^\alpha
\\
= &\int_{\mathbb{R}^d} \big (e^{i\omega\cdot x}-\sum_{j=0}^k{1\over j!}(i\omega\cdot x)^j\big )\hat{u}(\omega)d\omega.
\\
=&{\rm Re} \bigg ({i^{k+1}\over k!}\int_{\mathbb{R}^d} \int_{0}^T\left[(\bar \omega\cdot x - t)_+^ke^{i\|\omega\|t}
+(-1)^{k-1}(-\bar \omega\cdot x - t)_+^ke^{-i\|\omega\|_{\ell_1}t} \right]\hat{u}(\omega)\|\omega\|^{k+1}dt d\omega\bigg )
\\
=& {1\over k!}\int_{\{-1,1\}}\int_{\mathbb{R}^d} \int_{0}^T (z\bar \omega\cdot x - t)_+^k s(zt,\omega)  |\hat{u}(\omega)|\|\omega\|^{k+1}dtd\omega dz
\end{split}
\end{equation}
with $\int_{\{-1, 1\}} r(z) dz = r(-1) + r(1)$ and
\begin{equation} 
s(zt,\omega)= 
\begin{cases}
(-1)^{k+1\over 2}\cos(z\|\omega\|t + b(\omega)) & k \text{ is odd},
\\
(-1)^{k+2\over 2}\sin(z\|\omega\|t + b(\omega)) & k \text{ is even}.
\end{cases}
\end{equation} 
Define $G=\{-1,1\}\times [0,T]\times \mathbb{R}^{d}$, $\theta=(z, t, \omega)\in G$,
\begin{equation}\label{eq:straglam}
g(x,\theta)= (z\bar \omega\cdot x - t)_+^k {\rm sgn} s(zt,\omega),\qquad  \rho(\theta) = {1\over (2\pi)^d}|s(zt,\omega)||\hat{u}(\omega)|\|\omega\|^{k+1},\quad \lambda(\theta)={\rho(\theta)\over 
\|\rho\|_{L^1(G)}}.
\end{equation} 

Then \eqref{eq:fftaylor} can be written as 
\begin{equation}
u(x) = \sum_{|\alpha|\le k}{1\over \alpha!}D^\alpha u(0) x^\alpha
+ {\nu\over k!}\int_G  g(x, \theta)\lambda(\theta)d\theta,  
\end{equation}   
with $\nu=\int_G \rho(\theta)d\theta$. In summary, we have the following lemma.

\begin{lemma}\label{lm:probabilityexpan}
It holds that
\begin{equation}\label{ReLUm}
u(x) = \sum_{|\alpha|\le k}{1\over \alpha!} \partial^\alpha u(0) x^\alpha
+ {\nu\over k!}r_k(x),\qquad x\in \Omega
\end{equation}  
with $\nu=\int_G \rho(\theta)d\theta$ and 
\begin{equation}\label{ReLUrm}
r_k(x) = \int_G  g(x, \theta)\lambda(\theta)d\theta,\qquad G=\{-1,1\}\times [0,T]\times \mathbb{R}^{d},
\end{equation}  
and  $g(x,\theta)$, $\rho(\theta)$  and $\lambda(\theta)$ defined in \eqref{eq:straglam}.
\end{lemma}

According to \eqref{eq:straglam}, the main ingredient $(z\bar
\omega\cdot x - t)_+^k$ of $g(x,\theta)$ only includes the direction
$\bar\omega$ of $\omega$ which belongs to a bounded domain
$\mathbb{S}^{d-1}$. Thanks to the continuity of $(z\bar \omega\cdot x
- t)_+^k$ with respect to $(z, \bar\omega, t)$ and the boundedness of
$\mathbb{S}^{d-1}$, the application of the stratified sampling to the
residual term of the Taylor expansion leads to the 
approximation property in Theorem \ref{est:stratify}.

\begin{theorem}\label{est:stratify}
Assume $u\in B^{k+1}(\Omega)$
There exist $\beta_j\in [-1, 1]$, $\|\bar \omega_j\|=1$, $t_j\in [0,T]$ such that 
\begin{equation}
u_N(x)= \sum_{|\alpha|\le k}{1\over \alpha!} \partial^\alpha u(0) x^\alpha + {2\nu\over k!N}\sum_{j=1}^{N}\beta_j (\bar \omega_j\cdot x - t_j)_+^k
\end{equation} 
with $\nu=\int_G \rho(\theta)d\theta$ and $\rho(\theta)$  defined in \eqref{eq:straglam} 
satisfies the following estimate
\begin{equation}
\|u - u_N \|_{H^m(\Omega)} \lesssim  
\begin{cases}
 N^{-{1\over 2}-{1\over d}}\|u\|_{B^{k+1}(\Omega)},&m< k,
\\
N^{-{1\over 2}}\|u\|_{B^{k+1}(\Omega)}& m=k.
\end{cases} 
\end{equation} 
\end{theorem}

\begin{proof}
Let
$$
u_N(x)=  \sum_{|\alpha|\le k}{1\over \alpha!} \partial^\alpha u(0) x^\alpha + {\nu\over k!} r_{k,N}(x), \qquad r_{k,N}(x)={1\over N}\sum_{j=1}^{N}\beta_j (\bar \omega_j\cdot x - t_j)_+^k.
$$
Recall  the representation  of $u(x)$ in \eqref{ReLUm} and $r_k(x)$ in \eqref{ReLUrm}. It holds that
\begin{equation}
u(x) - u_N(x)={2\nu\over k!} (r_k(x) - r_{k,N}(x)).
\end{equation}
By Lemma \ref{lem:stratifiedapprox}, for any decomposition $\displaystyle G=\cup_{i=1}^N G_i$, there exist $\{\theta_i\}_{i=1}^N$ and $\{\beta_i\}_{i=1}^N\in [0, 1]$ such that 
\begin{equation}
\| \partial_x^\alpha (u - u_N)\|_{L^2(\Omega)} = {\nu\over k!}\|  \partial_x^\alpha (r_k - r_{k,N})\|_{L^2(\Omega)} \leq {1\over k!N^{1/2}}\max_{1\le j\le n}\sup_{\theta_{j},\theta_{j}'\in G_j} \|  \partial_x^\alpha \big(g(x,\theta_j) - g(x,\theta_j')\big)\|_{L^2(\Omega)}.
\end{equation}
Consider a $\epsilon$-covering decomposition $G=\cup_{i=1}^N G_i$  such that 
\begin{equation}
z=z',\ |t-t'|<\epsilon,\ \|\bar \omega  - \bar \omega'\|_{\ell^1}<\epsilon\qquad \forall \theta=(z, t, \omega),\ \theta'=(z', t', \omega')\in G_i
\end{equation}
where $\bar\omega$ is defined in \eqref{baromega}. 
For any $\theta_i, \theta'_i\in G_i$,  
$$
| \partial_x^\alpha \big (g(x,\theta_i) - g(x,\theta_i')\big )| = {k!\over (k-|\alpha|)!} | g_\alpha(x, \bar\omega, t) -  g_\alpha(x, \bar\omega', t')| 
$$
with 
\begin{equation}
 g_\alpha(x, \bar\omega, t)  = (z\bar \omega\cdot x-t)^{k-|\alpha|}_+\bar \omega^\alpha.
 \end{equation} 
 Since
$$
|\partial_{\bar\omega_i}  g_\alpha|\le (2T)^{m-|\alpha|-1}\big ((k-|\alpha|)x_i + 2T\alpha_i\big ), \qquad |\partial_t  g_\alpha|\le (k-|\alpha|)(2T)^{k-|\alpha|-1},
$$
it follows that
\begin{equation}
\big | \partial_x^\alpha \big (g(x,\theta_i) - g(x,\theta_i')\big )\big | \le {k!\over (k-|\alpha|)!}(2T)^{k-|\alpha|-1}   \bigg ( (k-|\alpha|)(|x|_{\ell_1}+1) + 2T|\alpha |\bigg ) \epsilon.
\end{equation}
Thus, by Lemma \ref{lem:stratifiedapprox}, if $m=|\alpha|<k$,
\begin{equation}
\|  \partial_x^\alpha (u - u_N)\|_{L^2(\Omega)} \le {|\Omega|^{1/2}\over (k-|\alpha|)!}(2T)^{k-|\alpha|-1}   \bigg ( (k-|\alpha|)(T+1) + 2T|\alpha |\bigg )N^{-{1\over 2}}\epsilon.
\end{equation}
Note that $\epsilon \sim N^{-{1\over d}}$. There exist $\theta_{i,j}$ such that for any $0\le k< m$,
\begin{equation}
\| u - u_N\|_{H^k(\Omega)} \le  C(m,k,\Omega)\nu N^{-{1\over 2}-{1\over d}}
\end{equation}
with $\nu\le \|u\|_{B^{k+1}(\Omega)}$ and
\begin{equation}\label{equ:defcmko}
C(m,k,\Omega)=|\Omega|^{1/2}\bigg (\sum_{|\alpha|\le k}{1\over (k-|\alpha|)!}(2T)^{k-|\alpha|-1}   \big ( (k-|\alpha|)(T+1) + 2T|\alpha |\big )\bigg )^{1/2}.
\end{equation} 
If $m=|\alpha|=k$,
$$
\max_{1\le j\le M}\sup_{\theta_{j},\theta_{j}'\in G_j} \| D_x^\alpha \big(g(x,\theta_j) - g(x,\theta_j')\big)\|_{L^2(\Omega)}\lesssim 1.
$$
This leads to 
\begin{equation}
\| u - u_N\|_{H^m(\Omega)} \le  C(m,k,\Omega)\nu N^{-{1\over 2}}\quad \mbox{for }\ k=m.
\end{equation}
Note that $u_N$ defined above can be written as
$$
u_N(x)=  \sum_{|\alpha|\le k}{1\over \alpha!} \partial^\alpha u(0) x^\alpha + {1\over k!N}\sum_{j=1}^{N}\beta_j (\bar \omega_j\cdot x - t_j)_+^k
$$ 
with $\beta_j\in [-1, 1]$,
which completes the proof.
\end{proof}

\begin{lemma}
There exist $\alpha_i$, $\omega_i$, $b_i$ and $N\le 2\begin{pmatrix} k+d\\k\end{pmatrix}$
such that
$$
 \sum_{|\alpha|\le m}{1\over \alpha!} \partial^\alpha u(0) x^\alpha = \sum_{i=1}^N\alpha_i (\omega_i\cdot x + b_i)_+^k
$$ 
with $
x^\alpha = x_1^{\alpha_1}x_2^{\alpha_2}\cdots x_d^{\alpha_d},\quad \alpha!=\alpha_1!\alpha_2!\cdots \alpha_d!.
$
\end{lemma}
The above result can be found in \cite{he2020preprint}

A combination of Theorem \ref{est:stratify} and the above the lemma gives the following estimate in Theorem \ref{th:stra}.
\begin{theorem} \label{th:stra}
Suppose $u\in B^{k+1}(\Omega)$.
There exist $\beta_j, t\in \mathbb{R}$, $\omega_j \in \mathbb{R}^d$ such that 
\begin{equation}
u_N(x)= \sum_{j=1}^{N}\beta_j (\bar \omega_j\cdot x - t_j)_+^k
\end{equation} 
satisfies the following estimate
\begin{equation}\label{d}
\|u- u_N \|_{H^m(\Omega)} \lesssim 
\begin{cases}
N^{-{1\over 2}-{1\over d}}\|u\|_{B^{k+1}(\Omega)},\qquad k> m,
\\
N^{-{1\over 2}}\|u\|_{B^{k+1}(\Omega)},\qquad k= m,
\end{cases}
\end{equation} 
where $\bar\omega$ is defined in \eqref{baromega}.
\end{theorem}

\begin{remark}
We make the following comparisons between results in Section \ref{sec:Bsplines} and \ref{sec:error2}.
\begin{enumerate}
\item The results in \ref{sec:Bsplines} are for activation functions $\sigma=b_k$, while the results in Section \ref{sec:error2} are for activation functions $\sigma={\rm ReLU}^k$.
\item By \eqref{splinetorelu}, the following relation obviously holds
$$
V_N(b_k)\subset V_{N+k}({\rm ReLU}^k).
$$
Thus, asymptotically speaking, the results that hold for $\sigma=b_k$ also hold for $\sigma={\rm ReLU}^k$.
\item The results in Section \ref{sec:error2} are  in some cases asymptotically better than those in Section \ref{sec:Bsplines}, but require more regularity assumptions on $u$.  For example, Theorem \ref{splinestratify} only requires $u\in B^m$, but Theorem  \ref{est:stratify} only requires $u\in B^{k+1}$ even for $m=0$.
\item The computational efficiency for  the solution of the optimization problems \eqref{m-mini-VN} or \eqref{min:uN} below, may be different with different choice of activation function, namely, $\sigma=b_k$ or ${\rm ReLU}^k$.
\end{enumerate}
\end{remark}

\section{Deep finite neuron functions, adaptivity and spectral
  accuracy}  
\label{sec:deep-fnm} 
In this section, we will study deep finite neural functions through
the framework of  deep neural networks and then discuss its adaptive
and spectral accuracy properties. 

\subsection{Deep finite neuron functions}
Given $d, \ell\in\mathbb{N}^+$, 
$
n_1,\dots,n_{\ell}\in\mathbb{N} \mbox{ with }n_0=d, n_{\ell+1}=1, 
$
\begin{equation}\label{thetamap}
\theta^i(x)=\omega_i\cdot x + b_i,\quad \omega_i\in \mathbb{R}^{n_{i+1}\times n_i},\ b\in \mathbb{R}^{n_{i+1}},
\end{equation}
and the activation function $\relu^k$, define
a  deep finite neuron function $u(x)$ from $\mathbb{R}^d$ to $\mathbb{R}$  as follows:
\begin{align*}
f^0(x)   &=\theta^0(x) \\ 
f^{i}(x) &= [  \theta^{i} \circ \sigma ](f^{i-1}(x)) \quad i = 1:\ell \\
f(x) &= f^\ell(x). \\
\end{align*}
The following more concise notation is often used in computer science literature:
\begin{equation}
\label{compress-dnn}
f(x) = \theta^{\ell}\circ \sigma \circ \theta^{\ell-1} \circ \sigma \cdots \circ \theta^1 \circ \sigma \circ \theta^0(x),
\end{equation}
here $\theta^i: \mathbb{R}^{n_{i}}\to\mathbb{R}^{n_{i+1}}$ are linear
functions as defined in \eqref{thetamap}.  Such a deep neutral network
has $(\ell+1)$-layer DNN, namely $\ell$-hidden layers. The size of
this deep neutral network is $n_1+\cdots+n_{\ell}$.

Based on these notation and connections, define deep finite neuron
functions with activation function $\sigma=\relu^k$ by
\begin{equation}
\label{NNL}
\dnn^k(n_1,n_2,\ldots, n_\ell)=\bigg\{ f^{\ell}(x) = \theta^\ell (x^{\ell}), 
 \mbox{ with } W^i\in \mathbb R^{n_{i+1}\times
	n_{i}}, b^i\in\mathbb R^{n_i}, i=0:\ell, n_0=d, n_{\ell+1}=1\bigg\}  
\end{equation}
Generally, we can define the $\ell$-hidden layer neural network as:
\begin{equation}
\dnn_\ell^k := \bigcup_{n_1, n_2, \cdots, n_{\ell}\ge 1} \dnn(n_1,n_2,\cdots,n_\ell, 1).
\end{equation}
For $\ell=1$, functions in $\dnn_1^k$ consist of piecewise polynomials of degree $k^2$ on a finite neuron grids whose boundaries are level sets of quadratic polynomials, see Fig \ref{fig:3}.
\begin{figure}[!ht]
\begin{center} 
\includegraphics[width=.3\textwidth]{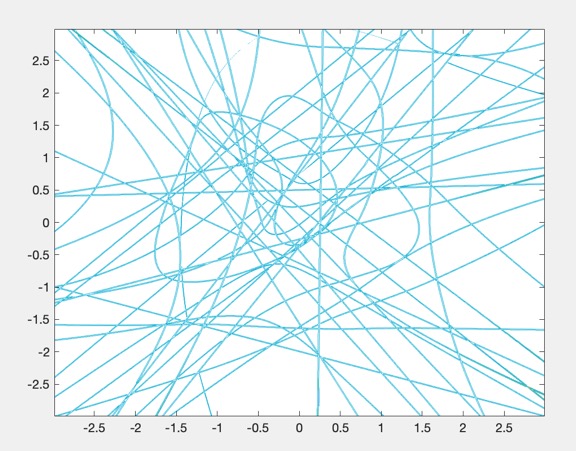}   
\includegraphics[width=.3\textwidth]{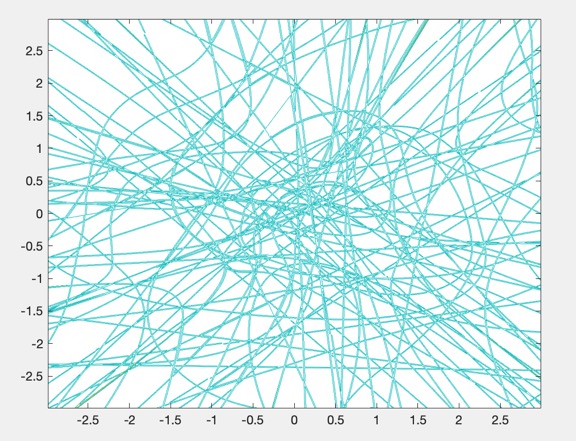}
\includegraphics[width=.28\textwidth]{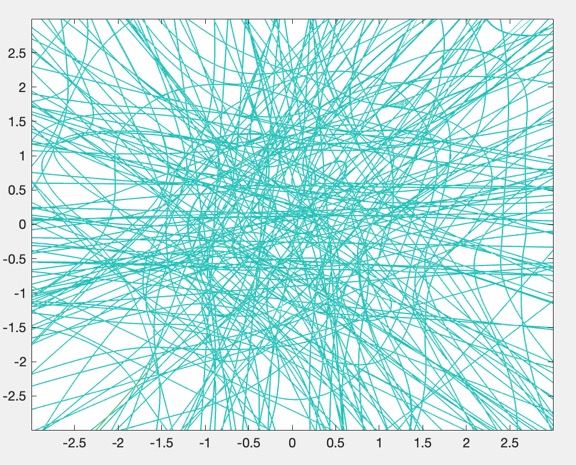}    
\caption{Hyperplanes with $\ell=2$}
\label{fig:3}
\end{center}
\end{figure} 

\subsection{Reproduction of polynomials and spectral accuracy}
One interesting property of the ReLU$^k$-DNN is that it reproduces
polynomials of degree $k$.  
\begin{lemma}
Given $k\ge 2$, $q\ge 2$, there exist $\ell\ge 1$, $n_1, \cdots, n_\ell$ such that
$$
\mathbb{P}_q\subset \dnn_l^k(n_1,\cdots,n_\ell),
$$
where $\mathbb{P}_q$ is the set of all polynomials with degree not larger than $q$.
\end{lemma}
For a proof of the above result, we refer to \cite{li2019better}. 

\begin{theorem}\label{thm:spectral}
Let  $\relu^k$ be the activation function, and $\dnn_\ell^k(N)$ be the DNN model with $\ell$ hidden layers. 
There exists some $\ell$ such that
\begin{equation}\label{eq:spectral}
\inf_{v_N\in \dnn_\ell^k(N)}\|u-v_N \|_{H^m(\Omega)} \lesssim
\inf_{v_N\in \mathbb{P}_{k^\ell}} \|u-v_N\|_{H^m(\Omega)},
\end{equation} 
\end{theorem}
Estimate \eqref{eq:spectral} indicates that the deep finite neuron function may provide spectral approximate accuracy.

\subsection{Reproduction of linear finite element functions and adaptivity}
The deep neural network with ReLU activation function have
been much studied in the literature and most widely used in practice.
One interesting fact is that ReLU-DNN is simply piecewise linear
functions.  More specifically, from \cite{he2018relu}, we have the
following result:

\begin{lemma}
Assume that ${\cal T}_h$ is a simplicial finite element grid of $N$ elements, in which any 
union of simplexes that share a same vertex is convex, any linear finite element function on this grid 
can be written as a ReLU-DNN with at most $\mathcal
  O(d)$ hidden layers. The number of neurons is at most
  $\mathcal{O}(\kappa^dN)$ for some constant $\kappa\ge 2$ depending
  on the shape-regularity of $\mathcal T_h$.  The number of non-zero
  parameters is at most $\mathcal{O} (d\kappa^dN)$.
\end{lemma}

The above result indicate that the deep finite neuron functions can
reproduce any linear finite element functions.  Given the adaptive
feature and capability of finite element methods, we see that the
finite neuron method can be at least as adaptive as finite element method.

\section{The finite neuron method for boundary value problems}\label{sec:convergence}
In this section, we apply the finite neuron functions for  numerical solutions of \eqref{2mPDE}. 
In \S \ref{sec:model}, we first present some analytic results for \eqref{2mPDE}. In \S \ref{sec:FNM}, we 
obtain error estimates for the finite neuron method for \eqref{2mPDE} for both the Neumann and 
Dirichlet boundary conditions. 
\subsection{Elliptic boundary value problems of order $2m$}\label{sec:model}
As discussed in the introduction, let us rewrite the Dirchlet boundary
value problem as follows:
\begin{equation} \label{m-BD}
\left\{
  \begin{array}{rccl}\displaystyle
Lu &=& f &\mbox{in }\Omega, \\
B_{D}^k(u) &= &0 & \mbox{on }\partial\Omega \quad(0\le k\le m-1).
  \end{array}
\right.
\end{equation}
Here $B_{D}^k(u)$ are given by \eqref{BD}. We next discuss about the pure Neumann boundary conditions 
for general PDE operator \eqref{Lu} when $m\ge 2$.  We first begin our discussion
with the following simple result. 
\begin{lemma} \label{lem:BDBNdual}
For each $k=0,1,\ldots,m-1$, there exists a bounded linear differential operator of order $2m-k-1$:
\begin{equation}
    \label{BN}
B_N^k: H^{2m}(\Omega)\mapsto L^2(\partial\Omega)    
  \end{equation}
such that the following identity holds
$$
(Lu,v)=a(u,v)-\sum_{k=0}^{m-1}\langle B_N^k(u),B_D^k(v)\rangle _{0,\partial\Omega}.
$$
Namely
\begin{equation}
\label{BDBNdual}
\sum_{|\alpha|=m}(-1)^m\left(\partial^\alpha
  (a_\alpha\,\partial^\alpha\,u),\,v\right) _{0,\Omega}
=\sum_{|\balpha|=m}\left(a_\alpha\partial^\balpha\,u, \partial^\balpha v\right) _{0,\Omega}
-\sum_{k=0}^{m-1}\langle B_N^k(u),B_D^k(v)\rangle _{0,\partial\Omega}
\end{equation}
for all $u\in H^{2m}(\Omega), v\in  H^{m}(\Omega)$. Furthermore, 
\begin{equation}\label{equ:regassum}
\sum_{k=0}^{m-1}\|B_D^k(u)\|_{L^2(\partial\Omega)}+\sum_{k=0}^{m-1}\|B_N^k(u)\|_{L^2(\partial\Omega)}\lesssim \|u\|_{2m, \Omega}.
\end{equation}
\end{lemma}
Lemma \ref{lem:BDBNdual} can be proved by induction with
respect to $m$.  We refer to \cite{lions2012non} (Chapter 2) and \cite{chen2020nonconforming} for a proof on a similar identity.

In general the explicit expression of $B_N^k$ can be quite
complicated.  Let us get some idea by looking at some simple examples
with the following special operator:
\begin{equation}
  \label{Delta-m}
Lu=(-\Delta)^m u+u,
\end{equation}
and 
\begin{equation}
a(u,v)= \sum_{|\alpha | = m}(a_\alpha\partial^{\alpha}u, \partial^{\alpha}v)_{0,\Omega} +(a_0u,v)\quad \forall
u, v \in V.
\end{equation}

\begin{itemize}
\item For $m=1$, it is easy to see that $B_N^0 u=\frac{\partial u}{\partial
    \nu}|_{\partial\Omega}$.
\item For $m=2$ and $d=2$, see \cite{chien1980variational}: 
$$
B_N^0 u= \frac{\partial}{\partial \nu}\left(\Delta
  u+\frac{\partial^2u}{\partial
    \tau^2}\right)-\frac{\partial}{\partial
  \tau}\left({\kappa_\tau}\frac{\partial u}{\partial \tau}\right)|_{\partial\Omega}~~\hbox{and}~~B_N^1 u=\frac{\partial^2u}{\partial \nu^2}|_{\partial\Omega},
$$
   with $\tau$ being the
anti-clockwise unit tangential vector, and $\kappa_\tau$ the curvature
of $\partial\Omega$. 
\end{itemize}


We are now in a position to state that the pure Neumann boundary value
problems for PDE operator \eqref{Lu} as follows.
\begin{equation} \label{m-BN}
\left\{
  \begin{array}{rccl}
Lu &=& f &\mbox{in }\Omega, \\
B_{N}^k(u) &= &0 & \mbox{on }\partial\Omega \quad(0\le k\le m-1).
  \end{array}
\right.
\end{equation}
Combining the trace theorem for $H^m(\Omega)$, see \cite{adams2003sobolev}, and Lemma
\eqref{lem:BDBNdual},  it is easy to see that 
\eqref{minJv} is equivalent to 
\eqref{m-BN} 
with $V=H^m(\Omega)$. 

For a given parameter $\delta>0$, we next consider the following
problem with mixed boundary condition: 
\begin{equation} \label{equ:delta}
\left\{
\begin{aligned}
Lu_{\delta} &= f \qquad \mbox{in }\Omega, \\
B_D^k(u_{\delta})+\delta B_N^k(u_\delta)  &= 0, \ \ 0\le k\le m-1.
\end{aligned}
\right.
\end{equation}
It is easy to see that \eqref{equ:delta} is equivalent to the following problem: Find $u_\delta\in H^m(\Omega)$, such that 
\begin{equation}\label{equ:varpdelta}
J_{\delta}(u_{\delta})=\min_{v\in H^m(\Omega)} J_{\delta}(v).
\end{equation}
where
\begin{equation}\label{J-delta}
J_\delta(v)={1\over 2}a_\delta(v,v)-(f,v)
\end{equation}
and
\begin{equation}\label{a-delta}
a_\delta(u,v)=a(u,v)+\delta^{-1}\sum_{k=0}^{m-1}\langle B_D^k(u), B_D^k(v)\rangle_{0,\partial\Omega}.
\end{equation}
In summary, we have 
\begin{lemma}
The following equivalences hod:
\begin{enumerate}
\item $u$ solves for \eqref{m-BD} or \eqref{m-BN} if and only if $u$ solves
\begin{equation}\label{m-mini}
J(u)=\min_{v\in V} J(v)
\end{equation}
 with $V=H^m_0(\Omega)$ or  $V=H^m(\Omega)$,
\item $u_\delta$ solves for \eqref{equ:delta}  if and only if $u_\delta$ solves 
$$
\displaystyle J_\delta(u_\delta)=\min_{v\in V} J_\delta(v)
$$ 
with $V=H^m(\Omega)$. 
\end{enumerate}
\end{lemma}

\begin{lemma} \label{lem:JvJu}
Assume that $u\in V$ be solution of \eqref{m-BD} or \eqref{m-BN} and $u_\delta\in V$ be the solution of \eqref{equ:varpdelta}, then
the following identities hold:
\begin{equation}\label{vuidentity1}
\|v-u\|_{a}^2=J(v)-J(u)\quad \forall v\in V,
\end{equation}
and
\begin{equation}\label{vuidentity2}
\|v-u_\delta\|_{a,\delta}^2=J_{\delta}(v)-J_{\delta}(u_{\delta})\quad \forall v\in V. 
\end{equation}
Here
\begin{equation}
  \label{a-norm}
\|v\|_{a}^2=a(v,v),\quad \|v\|_{a,\delta}^2=a_\delta(v,v).
\end{equation}
\end{lemma}
\begin{proof}
Let $u$ be the solution of \eqref{minJv}. Given $v\in V$, consider the quadratic function of $t$:
$$
g(t)=J(u+t(v-u)).
$$
It is easy to see that 
$$
0=\arg\min_{t}g(t), \quad g'(0)=0,
$$
and
$$
J(v)-J(u)=g(1)-g(0)=g'(0)+{1\over2}g''(0)=\|v-u\|_{a}^2.
$$
This completes the proof of \eqref{vuidentity1}. The proof of \eqref{vuidentity2} is similar. 
\end{proof} 

\begin{lemma}\label{Nitchtrick}
  Let $u$ be the solution of \eqref{m-BD} and $u_\delta$ be the solution of \eqref{equ:delta}. Then 
  \begin{align}\label{diff:uudelta}
\|u-u_\delta\|_{a,\delta} \lesssim\sqrt{\delta}  \|u\|_{2m,\Omega}.
\end{align}
\end{lemma}
\begin{proof}
Let $w=u-u_{\delta}$ and we have 

\begin{equation} \label{equ:diff}
\left\{
\begin{aligned}
Lw &= 0 \qquad \mbox{in }\Omega, \\
B_D^k(w)+\delta B_N^k(w)  &= \delta B_N^k(u),\ \ 0\le k\le m-1.
\end{aligned}
\right.
\end{equation}
By Lemma \ref{lem:BDBNdual}, and \eqref{equ:diff},  we have 
\begin{align}
0&=(Lw, w)
\\
&=\sum_{|\alpha|=m}(a_\alpha\partial^\alpha w, \partial^\alpha w)-\sum_{k=0}^{m-1}\int_{\partial \Omega}B_N^k(w)B_D^k(w) ds+(a_0w, w)
\\
&=\sum_{|\alpha|=m}(a_\alpha\partial^\alpha w, \partial^\alpha w)+ \sum_{k=0}^{m-1} \int_{\partial \Omega}(\delta^{-1}B_D^k(w)-  B_N^k(u) )B_D^k(w) ds+(a_0w, w),
\end{align}
implying
\begin{align}
&a(w, w)+ \delta^{-1}\sum_{k=0}^{m-1}\int_{\partial \Omega}B_D^k(w)^2 ds
=
\sum_{k=0}^{m-1}\int_{\partial \Omega}B_N^k(u)B_D^k(w) ds.
\end{align}
By Cauchy inequality, we have 
\begin{align}
a(w,w)+\delta^{-1}\sum_{k=0}^{m-1}\|B_D^k(w)\|^2_{L^2(\partial \Omega)}
\le& \sum_{k=0}^{m-1}\|B_N^k(u) \|_{L^2(\partial \Omega)}\|B_D^k(w)\|_{L^2(\partial \Omega)}
\\
\le& 2\delta  \sum_{k=0}^{m-1}\|B_N^k(u) \|^2_{L^2(\partial \Omega)}+\frac12 \delta^{-1} \sum_{k=0}^{m-1} \|B_D^k(w)\|^2_{L^2(\partial \Omega)},
\end{align}
which implies
\begin{align}
a(w,w)+\frac12 \delta^{-1}\sum_{k=0}^{m-1}\|B_D^k(w)\|^2_{L^2(\partial \Omega)}\le 2\delta  \sum_{k=0}^{m-1}\|B_N^k(u) \|^2_{L^2(\partial \Omega)}.
\end{align}
By the definition of $\|\cdot\|_{a,\delta}$ and noting that $w=u-u_\delta$, we have 
\begin{align}\label{diff:uudelta1}
\|u-u_\delta\|_{a,\delta}^2\le 4\delta  \sum_{k=0}^{m-1}\|B_N^k(u) \|^2_{L^2(\partial \Omega)}.
\end{align}
Combing this with \eqref{equ:regassum}, then completes the proof.
\end{proof}

\begin{lemma}\label{Regularity}
For any $s\ge -m$ and  $f\in H^s(\Omega)$, the solution $u$
  of \eqref{m-BD} or \eqref{m-BN} satisfies $u\in H^{2m+s}(\Omega)$ and 
  \begin{equation}
    \label{regularity}
\|u\|_{2m+s,\Omega}\lesssim     \|f\|_{s,\Omega}.
  \end{equation}
\end{lemma}
We refer to \cite{lions2012non} (Chapter 2, Theorem 5.1 therein) for a detailed proof. 

Following from \eqref{regularity} and \eqref{embend}, we have 
\begin{lemma} For any $s\ge -m$,  $\epsilon >0$, 
 and  $f\in H^s(\Omega)$, the solution $u$
  of \eqref{m-BD} or \eqref{m-BN} satisfies
  \begin{equation}
  \label{Barron-regularity}
\|u\|_{B^{m+1}(\Omega)}\le   \|f\|_{-m+\frac{d}{2}+1+\epsilon,\Omega}.
\end{equation}
\end{lemma}

\subsection{The finite neuron method for \eqref{2mPDE} and error estimates}\label{sec:FNM}
Let $V_N\subset V$ be a subset of $V$ defined by \eqref{VkN} which may not be linear
subspace. Consider the the discrete problem of \eqref{m-mini}:
\begin{equation}
  \label{m-mini-VN}
\mbox{Find $u_N\in V_N$ such that } J(u_N)=\min_{v_N\in V_N}J(v_N).
\end{equation}
It is easy to see that the solution to \eqref{m-mini-VN} always exists (for deep
neural network functions as defined below), but may not be unique.

\begin{theorem} Let $u\in V$ and $u_N\in V_N$ be solutions to
  \eqref{m-BN} and \eqref{m-mini-VN} respectively.   Then
  \begin{equation}
    \label{best-approx}
\|u-u_N\|_a=\inf_{v_N\in V_N}\|u-v_N\|_a.
  \end{equation}
\end{theorem}
\begin{proof}
By  Lemma~\ref{lem:JvJu}, we have
$$
\|u_N-u\|_a^2=J(u_N)-J(u)\le J(v_N)-J(u)=\|v_N-u\|_a^2,\quad\forall v\in V_N.
$$
The proof is completed.
\end{proof}
We obtain the following result.
\begin{theorem}
Let $u\in V$ and $u_N\in V_N$ be solutions to
\eqref{m-BN} and \eqref{m-mini-VN} respectively.  Then for arbitrary $\epsilon>0$, we have
\begin{equation}\label{error:N}
\|u-u_N\|_a \lesssim 
(\|f\|_{L^2(\Omega)}+\|f\|_{-k+\frac{d}{2}+1+\epsilon})
\begin{cases}
N^{-{1\over 2}-{1\over d}}&m< k,
\\
N^{-{1\over 2}}& m=k.
\end{cases} 
\end{equation}
\end{theorem}
By \eqref{best-approx}, Theorem \ref{est:stratify} and the embedding of Barron space into Sobolev space, namely Lemma \ref{smoothness-lemma}, the regularity result \eqref{regularity}, we get the proof. 

Next we consider the discrete problem of \eqref{equ:varpdelta}:
\begin{equation}
\label{min:uN}
\mbox{Find $u_N\in V_N$ such that } J_\delta(u_N)=\min_{v_N\in V_N}J_\delta(v_N).
\end{equation}

\begin{lemma}
For any given number $\delta$, let $u_\delta$ be the solution of \eqref{equ:delta} and
  $u_N$ be the solution of \eqref{min:uN}, respectively. We have
  \begin{equation}
  \label{eq:1}
\|u_N-u_\delta\|_{a,\delta} \lesssim (1+ \delta^{-\frac12})  \inf_{v_N\in
  V_N} \|v_N-u_\delta\|_{m,\Omega}. 
  \end{equation}
\end{lemma}
\begin{proof}
First of all, by Lemma~\ref{lem:JvJu} and the variational property, it holds that
\begin{align}\label{eq:min}
\|u_N-u_\delta\|_{a,\delta}^2=J_{\delta}(u_N)-J_{\delta}(u_{\delta})\le J_{\delta}(v_N)-J_{\delta}(u_{\delta})
=\|v_N-u_\delta\|_{a,\delta}^2, \quad \forall\,v_N\in V_N.
\end{align}
Further, for any $v_N\in V_N$, by the definition of $\|\cdot\|_{a,\delta}$ and trace inequality, we have 
$$
\|v_N-u_\delta\|_{a,\delta}\lesssim  \|v_N-u_\delta\|_{m,\Omega}+\delta^{-\frac12} \|v_N-u_\delta\|_{0,\partial\Omega}\lesssim (1+\delta^{-\frac12})\|v_N-u_\delta\|_{m,\Omega}.
$$
This completes the proof. 
\end{proof}

\begin{lemma}
  Let $u$ be the solution of \eqref{m-BD} and
  $u_N$ be the solution of \eqref{min:uN}, respectively. We have
  \begin{equation}
  \label{eq:13}
\|u_N-u\|_{a,\delta} \lesssim (1+ \delta^{-\frac12})  \inf_{v_N\in
  V_N} \|v_N-u\|_{m,\Omega}
+\sqrt{\delta}  \|f\|_{L^2(\Omega)}. 
  \end{equation}
\end{lemma}
\begin{proof}
First, by triangle inequality and \eqref{eq:min}, for any $v_N\in V_N$, we have
\begin{align}
\|u_N-u\|_{a,\delta}&\le  \|u_N-u_\delta\|_{a,\delta}+\|u_\delta-u\|_{a,\delta}\\
&\le \|v_N-u_\delta\|_{a,\delta}+\|u_\delta-u\|_{a,\delta}\\
&\le  \|v_N-u\|_{a,\delta}+2\|u_\delta-u\|_{a,\delta} .
\end{align}
Then by the definition of $\|\cdot\|_{a,\delta}$, trace inequality and \eqref{diff:uudelta}, for any $v_N\in V_N$, we have 
 \begin{equation}
\|u_N-u\|_{a,\delta} \lesssim (1+ \delta^{-\frac12}) \|v_N-u\|_{m,\Omega}
+\sqrt{\delta}  \|f\|_{L^2(\Omega)}. 
\end{equation}
This completes the proof. 
\end{proof}

\begin{theorem}
 Let $u$ be the solution of \eqref{m-BD} and
  $u_N$ be the solution of \eqref{min:uN} with $\delta\sim N^{-1/2 -1/{d}}$, respectively. Then
\begin{equation}
\label{error:D}
\|u-u_N\|_a \lesssim 
(\|f\|_{L^2(\Omega)}+\|f\|_{-k+\frac{d}{2}+1+\epsilon})
\begin{cases}
N^{-{1\over 4}-{1\over {2d}}}&m< k,
\\
N^{-{1\over 4}}& m=k.
\end{cases} 
\end{equation}  
\end{theorem}
 \begin{proof} Let us only consider the case that $k>m$. 
By Theorem \ref{est:stratify},
$$
\inf_{v_N\in V_N}\|u-v_N\|_{m,\Omega}\lesssim N^{-\frac12 -{1\over d}}\|u\|_{ B^{m+1}(\Omega)}.
$$
Thus, by \eqref{eq:1}
\begin{multline}
\label{eq:11}
\|u-u_N\|_{m,\Omega}\lesssim \|u_N-u\|_{a,\delta} \lesssim \delta^{-\frac12 } N^{-\frac12 -{1\over d}}\|u\|_{B^{m+1, q}(\Omega)}+\delta^{\frac12}\|f\|_{L^2(\Omega)}
\\
\le 
(\delta^{-\frac12}  N^{-\frac12 -{1\over d}}+\delta^{\frac12})(\|u\|_{ B^{m+1}(\Omega)}+\|f\|_{L^2(\Omega)}),\ \ \forall\,\delta>0.
  \end{multline}
Set $\delta\sim N^{-1/2 -1/{d}}$, and it follows that
\begin{equation}
\|u-u_N\|_{m,\Omega}\lesssim N^{-{1\over 4}-{1\over 2d}}(\|u\|_{B^{m+1}(\Omega)}+\|f\|_{L^2(\Omega)}).
\end{equation}
Now by the embedding of Barron space to Sobolev space, namely Lemma \ref{smoothness-lemma},  the regularity result \eqref{regularity},
the proof is completed. 
 \end{proof}
 We note that \eqref{min:uN} was studied in
 \cite{weinan2018deep} for $m=1$ and $k=3$.  Convergence analysis for
 \eqref{m-mini-VN} and \eqref{min:uN} seems to be new in this paper.
For other convergence analysis of DNN for numerical PDE, we refer to
\cite{shin2020on:arXiv:2004.01806} and
\cite{mishra2020enhancing,mishra2020estimates} for convergence analysis of PINN
(Physics Informed Neural Network). 

\section{Summary and discussions} \label{sec:Summary}

In this paper, we consider a very special class of neural network
function based on ReLU$^k$ as activation function.  This function
class consists of piecewise polynomials which closely resemble finite
element functions.  By considering elliptic boundary value problems of $2m$-th
order in any dimensions, it is still unknown how to construct
$H^m$-conforming finite element space in general in the classic finite
element setting.  In contrast, it is rather straightforward to
construct $H^m$-conforming piecewise polynomials using neural
networks, known as the finite neuron method,  and we further proved
that the finite neuron method provides good approximation
properties. 

It is still a subject of debate and of further investigation whether
it is practically efficient to use artificial neural network for
numerical solution of partial differential equations.  One major
challenge for this type of method is that the resulting optimization
problem is hard to solve, as we shall discuss below.

\subsection{Solution of the non-convex optimization problem}
\eqref{m-mini-VN} or \eqref{min:uN} is a highly nonlinear and
non-convex optimization problem with respect to parameters defining
the functions in $V_N$, see \eqref{VkN}. How to solve this type of
optimization problem efficiently is a topic of intensive research in
deep learning. For example,
stochastic gradient method is used in \cite{weinan2018deep} to solve
\eqref{min:uN} for $m=1$ and $k=3$.  Multi-scale deep neural network
(MscaleDNN) \cite{liu2020multi} and phase shift DNN (PhaseDNN)
\cite{cai2019phase} are developed to convert the high frequency
solution to a low frequency one before training. Randomized Newton's
method is developed to train the neural network from a nonlinear
computation point of view \cite{chen2019randomized}.  More refined
algorithms still need to be developed to solve \eqref{m-mini-VN} or
\eqref{min:uN} with high accuracy so that the convergence order,
\eqref{error:N} or \eqref{error:D}, of the finite neuron
method can not be achieved.

\subsection{Competitions between locality and global smoothness}
One insight gained from the studies in the paper is that the
challenges in constructing classic $H^m$-finite element subspace seems
to lie in the competitions between local d.o.f.  (degree of freedom)
and global smoothness.  In the classic finite element, one requires to
define d.o.f. on each element and then glue the local d.o.f. together
to obtain a globally $H^m$-smooth function. This process has proven to
be very difficult to realize in general when $m\ge 2$. But, if we relax the
locality, as in Powell-Sabine
element~\cite{powell1977piecewise}, we can use piecewise polynomials
of lower degree to construct globally smooth function. The neural
network approach studied in this paper can be considered as a global
construction without any use of a grid in the first place (even though
an implicitly defined grid exists). As a result, it is quite easy to
construct globally smooth functions that are piecewise polynomials. It
is quite remarkable that such a global construction
leads to function class that has very good approximation
properties. This is an attractive property of the function classes from
the artificial neural network.  One feasible question to ask if it is
possible to develop finite element construction technique that are
more global than the classic finite element but more  local than the
finite neuron method, which may be an interesting topic for
further research.

 \begin{center}
\begin{tabular}{|c|c|c|c|}
\hline
Local D.O.F.  &Slightly more global &$\cdots$ &global \\
\hline
General grid& Special grid& $\cdots$ &No grid\\
\includegraphics[width=0.15\textwidth]{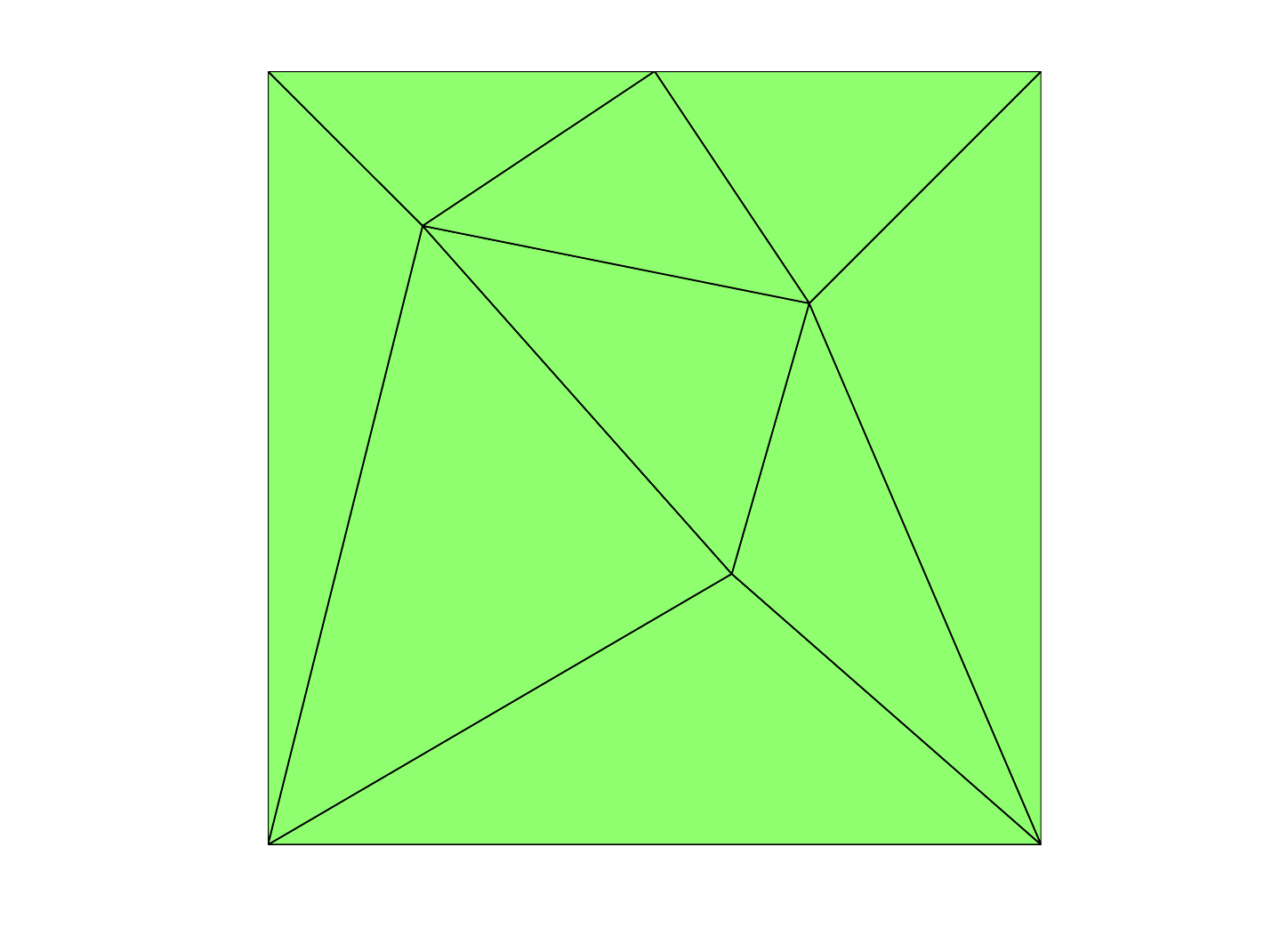} &\includegraphics[width=0.22\textwidth]{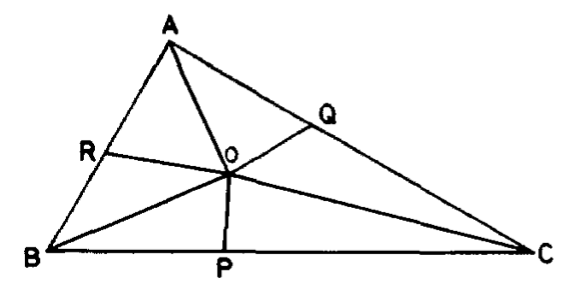} &$\cdots$ & \\
Conjecture: $k=(m-1)2^d+1$  &  Powell-Sabin \cite{powell1977piecewise}&  $\cdots$ & ReLu$^m$-DNN \\
True: $d=1,m\ge 1$ &    $k=2$ &   $\cdots$ & $k=m$     \\
$d=2,m=2$ (Still open)  & $d=2,m=2$ & $?$&  any $d$ and $m$ \\
\hline
\end{tabular}
\end{center}
Observation: More global d.o.f. lead to easier
construction of conforming elements for high order PDEs.

\subsection{Piecewise $P_m$ for $H^m(\Omega)$: from finite element to
  finite neuron method} \label{sec:concluding} As it is noted above,
in the classic finite element setting, it is challenging to construct
$H^m$-conforming finite element spaces for any $m, d\ge 1$.  But if we
relax the conformity, as shown in \cite{wang2013minimal}, it is
possible to give a universal construction of convergent
$H^m$-nonconforming finite element consisting of piecewise polynomial
of degree $m$.  In the finite neuron method setting, by relaxing the
constraints from the a priori given finite element grid, the
construction of $H^m$-conforming piecewise polynomials of degree $m$
becomes straightforward.  In fact, the finite neuron method can be
considered as mesh-less method, or even, vertex-less method although
there is a hidden grid for any finite neuron function.  This raises a
question if it is possible to develop some "in-between" method that have
the advantages of both the classic finite element method and the
finite neuron method. 

\subsection{Adaptivity and spectral accuracy} 
One of the important properties in the traditional finite element
method is its ability to locally adapt the finite element grids to
provide accurate approximation of PDE solution that may have local
singularities (such as corner singularities and interface
singularities).  In contrast, the traditional spectral method (using
high order polynomials) can provide very high order accuracy for
solutions that are globally smooth.  The finite neuron method analyzed
in this paper seems to possess both the adaptivity feature as in the
traditional finite element method and also the global spectral
accuracy as in the traditional spectral methods.  Adaptivity feature
of the finite neuron method is expected since, as shown in 
\S~\ref{sec:deep-fnm},  the deep finite neuron method can recover locally
adaptive finite element spaces for $m=1$.  Spectral feature of the
finite neuron method is illustrated in Theorem~\ref{thm:spectral}.
As a result, tt is conceivable that the finite neuron method may have both the
local and also global adaptive feature, or perhaps even adaptive
features in all different scales.  Nevertheless, such highly adaptive
features of the finite neuron method come with a potentially big
price, namely the solution of a nonlinear and non-convex optimization
problems.

\subsection{Comparison with PINN}
One important class of methods that is related to the FNM analyzed in
this paper is the the method of physical-informed neural networks
(PINN) introduced in \cite{raissi2019physics}.  By minimizing
certain norms of PDE residual together with penalizations of boundary
conditions and other relevant quantities, PINN is a very general
approach that can be directly applied to a wide range of problems.  In
comparison, FNM can only be applied to some special class of problems
that admit some special physical laws such as principle of energy
minimization or principle of least action, see \cite{feynmanfeynman}.
Because of the special physical law
associated with our underlying minimization problems, the Neumann
boundary conditions are naturally enforced in the minimization problem
and, unlike in the PINN method, no penalization is needed to enforce
such type of boundary conditions.

\subsection{On the sharpness of the error estimates}
In this paper, we provide a number of error estimates for our FNM such
as \eqref{straunbdd}, \eqref{d+1} and \eqref{d}, which give
increasingly better asymptotic order but also require more
regularities.  Even for sufficiently regular solution $u$, the best
asymptotic estimate \eqref{d} may still not be optimal.  In finite
element method, piecewise polynomial of degree $k$ usually give rise
to increasingly better asymptotic error when $k$ increases.  But the
asymptotic rate in the estimate of \eqref{d} does not improve as $k$
increases.  On the other hand, If $k>j$, ReLU$^k$-DNN
should conceivably give better accuracy than ReLU$^j$-DNN since
ReLU$^j$ can be approximated arbitrarily accurate by certain finite difference
of ReLU$^k$.  How to obtain better asymptotic estimates than \eqref{d} is
still a subject of further investigation.

We also note our error estimate \eqref{error:D} for Dirichlet boundary
condition is not as good as the one \eqref{error:N} for Neumann
boundary conditions.  This is undesirable and may not be optimal.  In
comparison, Nitsche trick does not suffer a loss of accuracy when used
in traditional finite element method.

\subsection{Neural splines in multi-dimensions}
The spline functions described in \S~\ref{sec:Bsplines} are widely
used in scientific and enginnering computing, but their generalization
multiple dimension are non-trivial, especially when $\Omega$ has
curved boundary. In \cite{hu2015minimal}, using the tensor product, the authors extended
the 1D spline to multi-dimensions on rectangular grids.  Some others
involve rational functions such as NURBS
\cite{cottrell2009isogeometric}. But the generalization of $\dnn_1^k$ or $\dnn_1(b^k)$
to multi-dimension is straightforward and also the resulting
(nonlinear) space has very good approximate properties.
 It is conceivable that the neural
network extension of B-spline to multiple dimensions which are locally
polynomials and globally smooth, may find useful applications in
computer aid design (CAD) and isogeometric analysis
\cite{cottrell2009isogeometric}. This is a potentially an interesting
research direction.

\section*{Acknowledgements}
Main results in this manuscript were prepared for and reported in
``International Conference on Computational Mathematics and Scientific
Computing'' (August 17-20, 2020,
http://lsec.cc.ac.cn/$\sim$iccmsc/Home.html).  and the author is
grateful to the invitation of the conference organizers and also to
the helpful feedbacks from the audience.  The author also wishes
to thank Limin Ma, Qingguo Hong and Shuo Zhang for their help in preparing
this manuscript.  This work was partially supported by the Verne
M. William Professorship Fund from Penn State University and the
National Science Foundation (Grant No. DMS-1819157).



\bibliographystyle{apalike}
\bibliography{IP_mn.bib,SX-references.bib,HaoRef.bib}
\end{document}